\documentclass[aos]{imsart}
\usepackage[utf8]{inputenc}

\usepackage{hyperref,amsmath,amssymb,natbib,enumerate,amsbsy,amsfonts, xcolor, graphics,graphicx, comment, algpseudocode, algorithm, bbm, booktabs, multirow,url, makecell, float, mathtools}
\RequirePackage{amsthm,amsmath,amsfonts,amssymb}
\usepackage[tableposition=above]{caption}
\usepackage{pifont}
\usepackage{tikz}

\usepackage{romannum}
\usepackage{pgfplots}
\usepackage{subcaption}

\usepackage{float}
\usepackage{amsmath}
\usepackage{amsthm}
\usepackage{amssymb}

\usepackage{color,soul}

\usepackage{cleveref}

\makeatletter

\newfloat{algorithm}{tbp}{loa}
\providecommand{\algorithmname}{Algorithm}
\floatname{algorithm}{\protect\algorithmname}

\theoremstyle{plain}
\newtheorem{thm}{\protect\theoremname}
\theoremstyle{remark}
\newtheorem{rem}[thm]{\protect\remarkname}
\theoremstyle{plain}
\newtheorem{cor}[thm]{\protect\corollaryname}
\theoremstyle{plain}
\newtheorem{prop}[thm]{\protect\propositionname}
\theoremstyle{plain}
\newtheorem{lem}[thm]{\protect\lemmaname}

\theoremstyle{definition}
\newtheorem{defi}{\protect\definitionname}

\usepackage{hyperref}       
\usepackage{url}            
\usepackage{booktabs}       
\usepackage{amsfonts}       
\usepackage{nicefrac}       
\usepackage{microtype}      
\usepackage{xcolor}         
\usepackage{soul}

\usepackage{mathtools}
\usepackage{dsfont}

\usepackage{enumitem}
\usepackage{breqn}
\usepackage{bbm} 
\usepackage{algorithm,algorithmicx,algpseudocode,caption}
\usepackage{pgfplots}



\DeclareMathOperator{\AND}{AND}

\allowdisplaybreaks

\global\long\def\s[#1]{\textnormal{\scriptsize #1}}
\global\long\def\st[#1]{\textnormal{\tiny #1}}

\global\long\def\P{\mathbb{P}}

\global\long\def\v[#1]{\mathbf{#1}} 
\global\long\def\m[#1]{\boldsymbol{#1}} 

\global\long\def\r[#1]{#1}

\global\long\def\dfn{:=}

\global\long\def\trre[#1,#2]{\overset{{\scriptstyle (#2)}}{#1}} 

\makeatother

\providecommand{\corollaryname}{Corollary}
\providecommand{\lemmaname}{Lemma}
\providecommand{\propositionname}{Proposition}
\providecommand{\remarkname}{Remark}
\providecommand{\theoremname}{Theorem}
\providecommand{\definitionname}{Definition}

\renewcommand\[{\begin{equation}}
\renewcommand\]{\end{equation}}

\begin{document}

\title{Adaptive Mean Estimation in the Hidden Markov sub-Gaussian Mixture Model}
\runtitle{Estimation in the Hidden Markov sub-Gaussian Mixture Model}
\begin{aug}
\author[A]{\fnms{Vahe}~\snm{Karagulyan}\ead[label=e1]{karagulyan@essec.edu}}
\and
\author[A]{\fnms{Mohamed}~\snm{Ndaoud}\ead[label=e2]{ndaoud@essec.edu}}
\address[A]{Department of Decisions Sciences, ESSEC Business School \printead[presep={,\ }]{e1,e2}}

\end{aug}
\begin{abstract}
We investigate the problem of center estimation in the high dimensional binary sub-Gaussian Mixture Model with Hidden Markov structure on the labels. We first study the limitations of existing results in the high dimensional setting and then propose a minimax optimal procedure for the problem of center estimation. Among other findings, we show that our procedure reaches the optimal rate that is of order $\sqrt{\delta d/n} + d/n$ instead of $\sqrt{d/n} + d/n$ where $\delta \in(0,1)$ is a dependence parameter between labels.
Along the way, we also develop an adaptive variant of our procedure that is globally minimax optimal. In order to do so, we rely on a more refined and localized analysis of the estimation risk. Overall, leveraging the hidden Markovian dependence between the labels, we show that it is possible to get a strict improvement of the rates adaptively at almost no cost.   
\end{abstract}
\maketitle

\textbf{Keywords:\/ sub-Gaussian mixtures, Hidden Markov Models, high dimensional mean estimation, adaptation.}

\pagenumbering{arabic}

\section{Introduction}
In the realm of statistical modeling and machine learning, Gaussian Mixture Models (GMMs) have emerged as a versatile and powerful tool. These models find their applications across a wide spectrum of fields, from data analysis to artificial intelligence. GMMs represent a class of probabilistic models that combine multiple Gaussian distributions, each characterized by its own set of parameters. A simple example of a GMM is the symmetric two-component Gaussian mixture model in $\mathbf{R}^d$ with centers $\pm \theta$:
$$\mathbf{P}_{\theta}=\frac{1}{2}\mathcal{N}(-\theta,\mathbf{I}_d)+\frac{1}{2}\mathcal{N}(\theta,\mathbf{I}_d),$$
where $\theta\in\mathbf{R}^d$ is the unknown center of the mixture.
Given a sample of $n$ identically distributed observations $(X_1,X_2,...,X_n) \sim \mathbf{P}_{\theta}$, the mixture can be represented in the following form. For all $i=1,2,...,n$, we have
$$X_i = \eta_i\theta + \xi_i,$$
where the labels $\eta_i\in \{-1;+1\}$, and $\xi_1,\xi_2,...,\xi_n$ are standard independent Gaussian random vectors in $\mathbf{R}^d$.
When dealing with data of Gaussian Mixture Models, two primary inference problems emerge: clustering and parameter estimation. On the one hand, the goal is to recover to which cluster each data point $X_i$ belongs, which can be seen as the problem of label  recovery up to a sign flip. There has been an extensive line of research on this problem, among them some works that studied the performance of Lloyd's algorithm \citep{Lu, ndaoud2022sharp}, that of spectral clustering \citep{abbe2022,loffler2021optimality} and that of SDP relaxation methods \citep{giraud2019partial, chen2021cutoff}, to name a few.
On the other hand, parameter estimation focuses primarily on estimating the center $\theta$. There are various works dealing with parameter estimation, for instance \citep{balak,raz1,raz2,raz3,raz4} and references therein. Most of the existing approaches evolve around variants of MLE such that the Expectation-Maximization (EM) algorithm (see also \citep{EM-klus} or \citep{EM-wu}). The latter proposes an algorithm which estimates $\theta$ under the quadratic loss (up to a sign flip) achieving the minimax error $\sqrt{\frac{d}{n}}\vee\frac{d}{n}$ up to a logarithmic factor, after $\sqrt{n}$ steps of expectation maximizations.

While most of the prior work focused on the case of independent observations i.e. independent labels. In this paper, we will be focusing on parameter estimation of a more sophisticated Gaussian Mixture Model with a Hidden Markovian structure, where observations have memory. Here, the label of each new observation depends on the previous one. More precisely $(\eta_i)_i$ is a Markov chain where every new label gets assigned the opposite sign of its preceding point with fixed probability $\delta \in \left(0,1/2\right)$. As $\delta$ varies between $0$ to $1/2$ the model interpolates between two better studied models, namely the Gaussian Location Model (when $\delta=0$) and the Gaussian Mixture Model (when $\delta=1/2$), which gives heuristics about the minimax rate of our model.  A formal description of this  model is given in Section \ref{sec: model}.

\subsection{Notation}
Throughout the paper we use the following notations. 
For given quantities $a_{n}$ and $b_n$, we write $a_{n}  \lesssim b_{n}$ ($a_{n} \gtrsim b_n$) when $a_{n} \leq c b_{n}$ ($a_{n} \geq c b_{n}$) for some absolute constant $c>0$. We write $a_{n} \sim b_{n} $ if $a_{n} \lesssim b_{n}$  and $a_{n} \gtrsim b_n$. For any $a,b \in \mathbf{R}$, we denote by $a\vee b$ ($a \wedge b$) the maximum (the minimum) of $a$ and $b$. The operator $\|.\|$ is the $\ell_{2}$-norm in $\mathbf{R}^d$, and $\|M\|_{op}$ is the operator norm of any matrix $M \in \mathbf{R}^{a\times b}$, with respect to the $\ell_{2}$-norm. $M^{\top}$ is used as the transpose of matrix $M$, and $\mathbf{I}_{d}$ the identity matrix of dimension $d$. For any symmetric matrix $A$, $\lambda_{max}(A)$ will denote the top eigenvalue of $A$  and $v_{max}(A)$ the corresponding unit eigenvector.  Finally $c_0$, $c_1$, $c$ are used for positive constants whose values may vary from theorem to theorem.

\subsection{Related work and our contributions}
Most of the relevant research on clustering mixture with a hidden markov chain structure has been done in the nonparametric setting as in \citep{gassiat2023model,ab_2023,leh2021,alex2015,decastro} and references therein. Our work is specific to the parametric Gaussian mixture model and our results are strongly related to those presented by \citep{yihan} who were, to the best of our knowledge, the first to provide statistical results for center estimation of the HMM (Hidden Markov Model) Gaussian Mixture Model, assuming the knowledge of flip probability $\delta$. More specifically they establish a minimax lower bound (Theorem 2) for the $\ell_2$ risk of center estimation. Moreover, they propose a procedure that is minimax optimal up to a logarithmic factor (Theorem 1). A brief summary of their algorithm is based on a divide-and-conquer routine. It starts by dividing the sample into $\ell$ blocks of length $k$ (depending on the mixing time of the Markov chain),  calculating the sample mean of each block separately, then applying a well chosen spectral algorithm on the corresponding means. The main focus here will be in the high dimensional regime where $d \lesssim n$. Since the case $d \gtrsim n$ is trivial as the Markovian structure does not seem to help and the rates do not depend on $\delta$ in this case. In what follows we always assume that $d \leq n$.

The lower bound provided in Theorem 1 \cite{yihan} can be summarized as follows
\begin{equation}\label{eq:lowr_bound}\underset{\tilde{\theta}}{\text{inf}} \;\underset{\|\theta\|^2 = M }{\text{sup}} \;\mathbf{E}_\theta(\ell^2(\tilde{\theta},\theta))\gtrsim
\begin{cases}
\sqrt{\frac{\delta d}{n}} + \frac{d}{n}, & M \leq \sqrt{\frac{\delta d}{n}} \vee \frac{d}{n} ,\\
\frac{\frac{\delta d}{n}}{M}, & \sqrt{\frac{\delta d}{n}}\vee \frac{d}{n} \leq M \leq \delta\vee \frac{d}{n},\\
\frac{d}{n}, & \delta\vee \frac{d}{n} \leq M, \\
\end{cases}\end{equation}
up to a logarithmic factor, where the loss function $\ell(\cdot,\cdot)$ is defined as
$$
\ell(\tilde{\theta},\theta)\dfn\min\{\|\tilde{\theta}-\theta\|,\|\tilde{\theta}+\theta\|\}.$$ 
Based on \eqref{eq:lowr_bound} it is straightforward that we also have 
$$
\underset{\tilde{\theta}}{\text{inf}} \;\underset{ \theta }{\text{sup}} \;\mathbf{E}_\theta(\ell^2(\tilde{\theta},\theta))\gtrsim \sqrt{\frac{\delta d}{n}} + \frac{d}{n}. 
$$

In order the study the performance of given estimators we define two notions of optimality as follows. Let us first define the rate function $\phi(\cdot)$ such that
$$
\forall \theta \in \mathbf{R}^d, \quad \phi(\theta):= \begin{cases}
\sqrt{\frac{\delta d}{n}} + \frac{d}{n}, & \|\theta\|^2  \leq \sqrt{\frac{\delta d}{n}} \vee \frac{d}{n} ,\\
\frac{\frac{\delta d}{n}}{\|\theta\|^2} , & \sqrt{\frac{\delta d}{n}}\vee \frac{d}{n} \leq \|\theta\|^2 \leq \delta\vee \frac{d}{n},\\
\frac{d}{n}, & \delta\vee \frac{d}{n} \leq \|\theta\|^2 . \\
\end{cases}
$$
\begin{defi}\label{definition:global}
Let $\hat{\theta}$ be a measurable estimator of $\theta$. We say that $\hat{\theta}$ is locally minimax optimal, if 
$$
\underset{\theta }{\text{sup}}\;\mathbf{P}\left(\ell^2(\hat{\theta},\theta)\gtrsim \phi(\theta)\right) \leq c_1\cdot e^{-  c_0 d}.
$$
Moreover, we say that $\hat\theta$ is globally minimax optimal, if 
$$
\underset{\theta }{\text{sup}}\;\mathbf{P}\left(\ell^2(\hat{\theta},\theta)\gtrsim \sqrt{\frac{\delta d}{n}} + \frac{d}{n}\right) \leq c_1\cdot e^{-  c_0 d}.
$$
\end{defi}
The choice of the deviation bound $c_1\cdot e^{-  c_0 d}$, in Definition \ref{definition:global}, is inherited from the simple case of the Gaussian location model $(\delta = 0)$, where the rate of estimation is the parametric rate $d/n$ and can be achieved with probability greater than $1 - c_1\cdot e^{-  c_0 d}$.

One may observe that local minimax optimality leads naturally to global minimax optimality and hence is a stronger notion of optimality. The lower bound \eqref{eq:lowr_bound} is reminiscent, although more general, of the one in \cite{EM-wu} who provide a similar lower bound in the independent case ($\delta=1/2$). The lower bound \eqref{eq:lowr_bound} is less pessimistic than the global minimax lower bound given by $\sqrt{ \delta d/n} + d/n$. Indeed, it shows that as $\|\theta\|$ grows the lower bound gets better and eventually we recover the parametric rate $d/n$ given enough separation between the clusters. The notion of local minimax optimality or scaled minimax optimality has also been studied by the second author in the context of linear regression \cite{ndaoud2019interplay, ndaoud2020scaled}.
We recall the reader here that when $d \gtrsim n$, the lower bound \eqref{eq:lowr_bound} reads as $d/n$ independently on $\delta$.

In terms of the upper bound, \cite{yihan} propose a procedure requiring only knowledge of $\delta$ and claim that it is locally minimax optimal in the sense that its risk matches the lower bound \eqref{eq:lowr_bound}. 
After investigating the proof of Theorem $1$, we found out what seems to be an issue in their argument. 
Alongside with this issue, their results sparked our curiosity about the possibility of constructing procedures achieving the minimax optimal rate, without additional logarithmic factors and with tighter probability bounds.

Our contributions can be summarized as follows:
\begin{itemize}
    \item Focusing on the case when flip probability $\delta$ of the model is known, we've improved the results of \cite{yihan} with a more sophisticated choice of the estimator $\hat{\theta}(\ell)$ where the number of buckets $\ell$ depends on both $\delta$ and $\|\theta\|$. As opposed to the procedure given by \cite{yihan}, we claim that the optimal procedure has to depend on $\|\theta\|$ as well.  We show that the risk of our procedure matches \eqref{eq:lowr_bound} without the logarithmic factor. Our results hold for sub-Gaussian noise, moreover the deviations of our bounds are exponential.
Formally for $\ell^*=d \vee \lceil n (\|\theta\|^2 \vee \delta )\rceil \wedge n$, we show that $\hat{\theta}(\ell^*)$ is locally minimax optimal and that
$$\ell^2(\hat{\theta}(\ell^*),\theta)\lesssim
\begin{cases}
\sqrt{\frac{\delta d}{n}} + \frac{d}{n}, & \|\theta\|^2 \leq \sqrt{\frac{\delta d}{n}}\vee \frac{d}{n}\\
\frac{\frac{\delta d}{n}}{\|\theta\|^2} , & \sqrt{\frac{\delta d}{n}}\vee \frac{d}{n} \leq \|\theta\|^2 \leq \delta\vee \frac{d}{n}\\
\frac{d}{n}, & \delta\vee \frac{d}{n} \leq \|\theta\|^2, \\
\end{cases}$$
with probability $1-c_1 e^{-c_0 d}$.

\item 
Furthermore, in Section \ref{sec: adaptive}, we propose an adaptive estimator $\hat{\theta}$ that only depends on $\delta$  which is also minimax optimal. Along the way, we propose heuristics that estimation of $\delta$ might be more complicated that estimation of $\theta$ in the high dimensional setting which suggests that fully adaptive procedures should not rely on estimating $\delta$ first.  

Our $\|\theta\|$-adaptive procedure is a two-stage algorithm. We start by constructing an estimator using the sub-optimal number of buckets $\ell_1=d \vee \lceil \delta n \rceil $ which leads to an estimator $\hat{s}=\|\hat{\theta}(\ell_1)\|$ of the norm $\|\theta\|$. We then update the number of buckets such that $\ell_2=d \vee \lceil n (\hat{s}^2  \vee \delta )\rceil\wedge n$ before returning the new estimator $\hat{\theta}(\ell_2 )$. Hence we show that adaptation to $\|\theta\|$ comes at no cost in this case and that local minimax optimality is possible as long as $\delta$ is known.

\item Finally, we focus on the full adaptive case. In order to do so, we propose another method that depends weakly on $\delta$ in Section \ref{sec:fully adaptive}. This procedure is not locally minimax optimal but achieves the global minimax rate $\sqrt{\delta d/n} + d/n$. Because of its week dependence on $\delta$, full adaptation is possible in that case. We suggest a final procedure inspired by Lepski's adaptation scheme on the number of buckets $\ell$ that is globally minimax optimal (up to a logarithmic factor) but not locally minimax optimal. 

\item As a summary, our fully adaptive procedure reaches the rate $\sqrt{\delta d/n} +d/n$ in the worst case scenario which leads to a strict improvement over the vanilla spectral method that cannot get a rate better than $\sqrt{d/n} + d/n$. It remains an open question whether one can find fully adaptive procedures that are also locally minimax optimal.
\end{itemize}
All the proofs are deferred to the Appendix.

\subsection{Model}{\label{sec: model}}

We consider the high-dimensional Binary Markov sub-Gaussian Mixture Model. In this context we observe $n$ samples of a $d$-dimensional vector $\theta$ multiplied by a random sign $\eta_i \in \{-1,1\}, (i=1,...,n)$ corrupted with an additive sub-Gaussian noise. 
Formally the model takes the following form: 
\begin{equation}
Y_i = \eta_i\theta + \xi_i, \label{eqn: model}
\end{equation}
where $Y_i\in\mathbf{R}^d$ are the observations, $\theta \in \mathbf{R}^d , d\geq1$, $\eta_i$-s follow a homogeneous binary symmetric Markov chain with flip probability $\delta\in(0,1/2)$, namely $\eta_1$ is a Rademacher random variable and for all $i$ 
$$\eta_{i+1} = \begin{cases}
\text{ }\text{ }\eta_i & \text{with prob. } 1-\delta\\
-\eta_i & \text{with prob. } \delta,
\end{cases}
$$
and $\xi_i$ are i.i.d. isotropic $1$-(sub)Gaussian random vectors that are independent from $\eta$. We remind the reader that a random vector $\omega\in \mathbf{R}^d$ is said to be $1$-(sub)Gaussian if and only if for all $v\in \mathbf{R}^d$ we have $\mathbf{E}(\exp(\langle v,\omega\rangle)) \leq \exp(\|v\|^2/2)$.

For our further analysis we define $X_{i+1}:=\eta_{i+1}\eta_{i}$ or equivalently $\eta_{i+1}=X_{i+1}\eta_i$. Hence 
$$X_{i} = \begin{cases}
\text{ }\text{ }1 & \text{with prob. } 1-\delta\\
-1 & \text{with prob. } \delta.
\end{cases}
$$Notice that each $X_{i+1}$ is independent from $\eta_j$ for all $j \leq i$, since the sign change doesn't depend on the previous state. Moreover $X_i$-s are i.i.d. We remind the reader that our goal is to construct optimal procedures with respect to the risk defined by $\ell^2(\cdot,\cdot)$.

\section{Improved mean estimation for known $\delta$}

Let's assume that the flip probability $\delta$ is known and that $\delta\in(0,\frac{1}{2})$ without loss of generality, since otherwise we can multiply samples with even indices by $-1$ and replace the flip probability by $1-\delta$. We begin with a modified procedure initially inherited from \cite{yihan}, dividing the sample into $\ell$ blocks of length $k$ (we'll call them buckets). We will assume without loss of generality that $n=k\ell$. We denote by $I_i$ the set of indices of bucket $i$. For each bucket $i=1,2,...,\ell$ consider the sample mean of $k$ observations inside the bucket. Namely    

$$\tilde{Y_i}=\frac{1}{k}\sum_{j\in I_i} \eta_j\theta+\frac{1}{k}{\sum_{j\in I_i}\xi_i},$$
which can be rewritten neatly as

$$\tilde{Y_i}= \bar{\eta}_i\theta+\frac{w_i}{\sqrt{k}},$$
where $\bar{\eta}_i=\frac{1}{k}\sum_{j\in I_i} \eta_j$ and $w_i$ is still an isotropic $1$-(sub)Gaussian vector.
We combine the $\ell$ relations in an $\mathbf{R}^{d\times\ell}$ matrix form as follows
\begin{equation}
    \tilde{Y}=\theta\bar{\eta}^T+\frac{w}{\sqrt{k}}.
\end{equation}
Averaging reduces the variance of the noise while weakening the signal where $\eta$ is now replaced by $\bar{\eta}$. 
A key observation here is that $\bar{\eta}_i^2$-s are i.i.d., since each  $\bar{\eta}_i^2$ only depends on $X_j$-s of the $i$-th bucket. Let us define $g(\delta):= \mathbf{E}(\bar{\eta}_i^2)$. The following Lemma gives a hint on the mixing time of the Markov chain.
\begin{lem}\label{lem:expecation}
    Let $\delta,\delta' \in (0,1/2)$, then we have that 
    $$
    |g(\delta) - g(\delta')| \leq \frac{2k|\delta-\delta'|}{3}.
    $$
\end{lem}
In particular, and for $\delta'=0$, it follows that $\mathbf{E}(\bar{\eta}_i^2) \geq 1 - 2k\delta/3$. And hence as long as $k\leq 1/\delta$, $\mathbf{E}(\bar{\eta}_i^2)$ is of constant order and the signal is preserved. In other words, as long as $k\leq 1/\delta$,  most of $\eta_i$ are equal within the same bucket with non-zero probability, hence we can average them without a loss in the signal while reducing the variance of the noise.

Consider the Gram matrix of observations $\tilde{Y}\tilde{Y}^T$ given by
$$\tilde{Y}\tilde{Y}^T=\theta\theta^T\|\bar{\eta}\|^2+\frac{w\bar{\eta}\cdot \theta^T+\theta \cdot(w\bar{\eta})^T}{\sqrt{k}}+\frac{1}{k}ww^T$$
and its expected value 
\[
\mathbf{E}[\tilde{Y}\tilde{Y}^T]=\theta\theta^T\mathbf{E}\|\bar{\eta}\|^2+\frac{\ell}{k}\mathbf{I}_d.
\]
For convenience denote $\Sigma \coloneqq \mathbf{E}\left(\frac{1}{\ell}\tilde{Y}\tilde{Y}^T\right)
$ and $\hat{\Sigma} \coloneqq \frac{1}{\ell}\tilde{Y}\tilde{Y}^T$.
Note that $\theta$ is proportional to the top eigenvector of $\Sigma$, meaning that $\lambda_{max}(\Sigma)=\|\theta\|^2\mathbf{E}\|\bar{\eta}\|^2/\ell+\frac{1}{k}$ and the corresponding unit eigenvector $v_{max}(\Sigma)=\frac{\theta}{\|\theta\|}$. Consequently, we consider the following estimator:

\[\label{def: hat_theta2}\hat{\theta}(\ell) \coloneqq \sqrt{\frac{\ell}{\mathbf{E}\|\bar{\eta}\|^2}\left(\lambda_{max}(\hat{\Sigma})-\frac{1}{k}\right)_+} \cdot v_{max}(\hat{\Sigma}).\]
\begin{thm}
\label{thm: known delta loss bound} 
For any values of $\ell \leq n$, the estimator $\hat{\theta}(\ell)$ satisfies 
\begin{equation}\label{eq:technical rates known}
\ell(\hat{\theta}(\ell),\theta) \lesssim
\begin{cases}
\frac{\ell}{\mathbf{E}\|\bar\eta\|^2}\left(\sqrt{\frac{d}{\ell}}\|\theta\| +\sqrt{\frac{d}{n}}
+ \frac{1}{\|\theta\|}\left(\sqrt{\frac{d}{k n}}+ \frac{d}{n}\right)\right) , & \|\theta\|^2 \leq 1\\
\sqrt{\frac{d}{n}} + \frac{d}{n}, & \|\theta\|^2 \geq 1 
\end{cases}
\end{equation}
 with probability greater than $1-c_1\cdot e^{-c_0d}$, and where we take $\ell = n$ in the case $\|\theta\|^2 \geq 1$. 
\end{thm}
Next proposition shows that the lower bound in \cite{yihan} is indeed optimal and can be reached using a fine-tuned estimator of the form $\hat{\theta}(\ell)$. To do that, we consider several values of $\ell$ depending on each particular scenario of $\|\theta\|$, and hence the optimal procedure depends on $\|\theta\|$ as opposed to what was claimed in \cite{yihan}. We  define $\ell^*$ as follows 
\begin{equation}\label{eq:ell*}
    \ell^*:= d \vee \lceil n( \delta \vee \|\theta\|^2) \rceil \wedge n.
\end{equation}

\begin{prop}\label{prop: theoretical rates}
    Assume that $d\leq n$. With probability greater than $1-c_1\cdot e^{-c_0d}$, the estimator $\hat{\theta}(\ell^*)$ with $\ell^*$ defined in \eqref{eq:ell*} satisfies

\[
\ell^2(\hat{\theta}(\ell^*),\theta)\lesssim
\begin{cases}
\sqrt{\frac{\delta d}{n}} + \frac{d}{n}, & \|\theta\|^2 \leq \sqrt{\frac{\delta d}{n}}\vee \frac{d}{n},\\
\frac{\frac{\delta d}{n} }{\|\theta\|^2} , & \sqrt{\frac{\delta d}{n}}\vee \frac{d}{n} \leq \|\theta\|^2 \leq \delta \vee \frac{d}{n},\\
\frac{d}{n}, & \delta \vee \frac{d}{n} \leq \|\theta\|^2. \\
\end{cases}
\label{eqn: rate with delta}
\]
Moreover $\ell^2(\hat{\theta}(\ell^*),\theta)\lesssim \sqrt{\frac{\delta d}{n}}+\frac{d}{n}$, with probability greater than $1-c_1\cdot e^{-c_0 d}$. 
\end{prop}
\begin{rem}
 The global minimax rate happens around the value $\|\theta\|^2 = \sqrt{\frac{\delta d}{n}}\vee \frac{d}{n}$ in  \eqref{eqn: rate with delta}, which corresponds to the worst case scenario. It is also clear that for $\delta \leq d/n$, the upper bound in \eqref{eqn: rate with delta} corresponds to the parametric rate of estimation $d/n$ and hence we do not get further improvement in estimation for $\delta \leq d/n$. In summary, the dependence helps get better estimation rates only in the regime $d/n \leq \delta < 1/2$.
\end{rem}

\begin{rem}
    The choice $k = \lfloor \frac{1}{\delta} \rfloor $ (or $\ell= \lceil n\delta \rceil$) is reasonable, since it is aligned with the mixing time of the Markov Chain $\eta$ and that is of order $1/\delta$. In other words, within each buckets the labels are more likely to share the same sign as long as the bucket size is smaller than $1/\delta$. So we may want to choose the smallest $\ell$ such that $\ell \geq n\delta$. Based on \eqref{eq:technical rates known}, it seems that the error term due to the misspecification around $\mathbf{E}\|\bar{\eta}\|$ gets worse as $\ell$ gets smaller (especially for large values of $\|\theta\|$). This explains why the choice of $\ell$ needs to depend both on $\delta$ and $\|\theta\|$.
\end{rem}

Here $\hat{\theta}(\ell^*)$ carries a theoretical interest as it shows that local minimax optimality is possible. It also lays a foundation for constructing adaptive procedures. 
Observing the proof one can point out that the optimal choice of $\ell^*$ depends on the signal strength $\|\theta\|$ and flip probability $\delta$. 
In practice, neither do we have knowledge about the signal strength $\|\theta\|$, as  $\theta$ is unknown, nor do we know $\delta$. So the first step to progress is to construct a procedure without assuming any knowledge about $\theta$, which is going to be the goal of next section.

\section{Adaptive algorithm for known $\delta$ }{\label{sec: adaptive}}

In order to build a semi adaptive procedure, it is very natural to consider plug-in estimators of $\|\theta\|$. In what follows we take the most reasonable approach, which is estimating $\|\theta\|$ with $\hat{s}$ such that
\begin{equation}\label{hat s}
    \hat{s}:=\|\hat{\theta}(d\vee \lceil n\delta \rceil \wedge n)\|.
\end{equation} First, we begin with a simple Lemma and a corresponding Corollary which provides a basis for our $\theta$ adaptive procedure.
\begin{lem}
\label{lem: s and theta}

Suppose $\ell= d \vee \lceil n\delta\rceil \wedge n $. Then there exists a constant $C$ large enough such that with probability $1-c_1e^{-c_0d}$ we get the following:
\begin{itemize}
    \item If $\delta\vee Cd/n\leq\|\theta\|^2 $, then

    \[\frac{\|\theta\|}{2}\leq\hat{s}\leq\frac{3\|\theta\|}{2} \label{eqn: order of s},\]
    \item and if $\delta \vee Cd/n > \|\theta\|^2 $, then \[\hat{s}^2 < 3(\delta \vee Cd/n) \label{eqn:less3delta}.\]
\end{itemize}

\end{lem}
As a result, we derive the next Corollary that shows that comparing $\hat{s}$ to $\delta\vee Cd/n$ is similar to comparing $\|\theta\|$ with $\delta\vee Cd/n$. 

\begin{cor}\label{cor: 3delta}
Suppose that $\ell= d \vee \lceil n \delta \rceil \wedge n$. Then for some $C>0$ large enough, with probability $1-c_1e^{-c_0d}$, we have the following:
\begin{itemize}
    \item If $\hat{s}^2 <  3(\delta\vee Cd/n)$ we get that $\|\theta\|^2<12(\delta\vee Cd/n),$
    \item and if $\hat{s}^2 \geq  3(\delta\vee Cd/n)$ we get that $ \|\theta\|^2 \geq \delta \vee Cd/n$.
\end{itemize}
\end{cor}
\begin{proof}
We shall treat the two cases separately. 
Consider first the case $\hat{s}^2 <  3(\delta\vee Cd/n)$.   If $\delta\vee Cd/n \leq \|\theta\|^2$ then \eqref{eqn: order of s} holds, and we get $\|\theta\|^2\leq4\hat{s}^2 < 12(\delta\vee Cd/n)$. If $\|\theta\|^2<\delta \vee Cd/n$, then obviously $\|\theta\|^2 < 12(\delta\vee Cd/n)$ as well. As for the second case $\hat{s}^2\geq 3(\delta\vee Cd/n)$, then we must have $\|\theta\|^2 \geq \delta \vee Cd/n$, otherwise we get a contradiction with \eqref{eqn:less3delta}.
\end{proof}

Having various permissions to replace $\|\theta\|$ with $\hat{s}$, we give an adaptive 2-step algorithm. The pseudocode of the algorithm is given in Algorithm \ref{alg: known delta}.

\begin{algorithm}[!ht]
\caption{Adaptive mean estimation for known $\delta$ \label{alg: known delta}}
\begin{algorithmic}[1]
\item[]
\State \textbf{input:}  Observations $Y_1,Y_2,...,Y_n$ from the model \eqref{eqn: model}, and flip probability $\delta$.
\item[]
\State $\ell_1 \gets d \vee \lceil \delta n \rceil \wedge n $
\item[]
\State $\hat{s} \gets \sqrt{\frac{\ell_1}{\mathbf{E}\|\bar{\eta}\|^2}\left(\lambda_{max}(\hat{\Sigma})-\frac{1}{k}\right)_+} $
\item[]
\State $\ell_2 \gets d \vee \lceil n( \delta \vee \hat{s}^2) \rceil \wedge n$
\item[]
\State $\hat{\theta}(\ell_2) \gets \sqrt{\frac{\ell_2}{\mathbf{E}\|\bar{\eta}\|^2}\left(\lambda_{max}(\hat{\Sigma})-\frac{1}{k}\right)_+} \cdot v_{max}(\hat{\Sigma})$
\item[]
\State\Return $\hat{\theta}=\hat{\theta}(\ell_2)$
\item[]
\end{algorithmic}

\end{algorithm} 

In  what follows Theorem \ref{thm: adaptive errors} confirms that our semi adaptive procedure is locally minimax optimal. The proof follows a similar analysis done in Proposition \ref{prop: theoretical rates} and can be found in the Appendix. We will now define the adaptive choice of  $\ell$ as follows 
\begin{equation}\label{eq:hat ell}
    \hat{\ell}:= d\vee  \lceil n( \delta \vee \hat{s}^2)  \rceil \wedge n.
\end{equation}

\begin{thm}
\label{thm: adaptive errors}
When $d\leq n$, the proposed estimator $\hat{\theta}(\hat{\ell})$ with $\hat{\ell}$ defined in \eqref{hat s}-\eqref{eq:hat ell} satisfies 
\[
\ell^2(\hat{\theta}(\hat{\ell}),\theta)\lesssim
\begin{cases}
\sqrt{\frac{\delta d}{n}} +\frac{d}{n}, & \|\theta\|^2 \leq \sqrt{\frac{\delta d}{n}}\vee \frac{d}{n},\\
\frac{\frac{\delta d}{n}} {\|\theta\|^2} , & \sqrt{\frac{\delta d}{n}}\vee \frac{d}{n} \leq \|\theta\|^2 \leq \delta \vee \frac{d}{n},\\
\frac{d}{n}, & \delta \vee \frac{d}{n} \leq \|\theta\|^2, \\
\end{cases}
\label{eqn: rate with delta2}
\]
with probability larger than $1-c_1\log(n/d) e^{-c_0d}$.
\\Globally $\ell^2(\hat{\theta}(\hat{\ell}),\theta)\lesssim \sqrt{\frac{\delta d}{n}} + \frac{d}{n}$, with probability larger than $1-c_1\log(n/d) e^{-c_0d}$. 
\end{thm}

\begin{rem}
    The result of Theorem \ref{thm: adaptive errors} holds with a slightly weaker probability. The logarithmic loss is due to the dependence between $\hat{\Sigma}$ and $\hat\ell$. One way to get rid of that dependence is by splitting the main sample into two equal sub-samples, and then using one for estimation of $\hat{\ell}$ and the other one for the construction of $\hat\theta(\hat\ell)$.
\end{rem}
We conclude that minimax optimal rates can be achieved adaptively to the signal $\theta$ both locally and globally. The optimal procedure depends only on $\delta$ through the choice of $\hat{\ell}$ and $g(\delta)$.

\section{Mean estimation for unknown $\delta$}\label{sec:fully adaptive}

Taking a closer look at our estimator $\hat{\theta}(\ell)$ we recall that it depends on $g(\delta)=\mathbf{E}(\|\bar\eta\|^2)$, which can be theoretically computed with the knowledge of the flip probability $\delta$. At this stage it might be tempting to replace $\delta$ by some plug-in estimator $\hat{\delta}$. In order to estimate $\delta$ observe that
$$
\mathbf{E}(\eta_{1}\eta_{2}) = (1-2\delta)\|\theta\|^2.
$$
Under Gaussian noise, we can show, as in \cite{yihan}, that a lower bound of estimation for $\delta$ is given by $1/\sqrt{n}$ for any $\|\theta \|\leq 1$. While this lower bound is far from being optimal it shows that in general we can not expect to get $\hat{\delta}$ such that $|\hat\delta - \delta | \leq  \delta/2$. This heuristic explains that using plug-in estimators of $\delta$ might not lead to optimal rates. 

An alternative approach is to completely bypass estimation of $\delta$. We start by proposing a modified version of the estimator $\hat{\theta}(\ell)$ where we replace $g(\delta)$ by $1$ which in turn leads to a weaker dependence on $\delta$. This choice is equivalent to considering the new signal $\Sigma$ such that
$$\Sigma=\theta\theta^T+\frac{\mathbf{I}_d}{k},
$$
and $$\hat{\Sigma}=\frac{1}{\ell}YY^\top=\theta\theta^T + \theta\theta^T\left(\frac{\|\bar{\eta}\|^2}{\ell} -1\right)+\frac{(w\bar{\eta})\theta^T+\theta(w\bar{\eta})^T}{\sqrt{k}\ell}+\frac{ww^T}{k\ell}.
$$
This motivates our new procedure $\tilde{\theta}(\ell)$ defined by $$\tilde{\theta}(\ell)=\sqrt{\left(\lambda_{max}(\hat{\Sigma})-\frac{1}{k}\right)_+} \cdot v_{max}(\hat{\Sigma}),$$
where the dependence on $\mathbf{E}\|\bar\eta\|^2$ was dropped. The next Theorem highlights  the cost of this operation. 
\begin{thm}\label{thm: unkown delta first loss bound}
For any value of $\ell\leq n$,  with probability greater than $1-c_1\cdot e^{-c_0d}$, the estimator $\tilde{\theta}(\ell)$ achieves the following:
\begin{equation}\label{eq:okok}
\ell(\tilde{\theta}(\ell),\theta) \lesssim
\begin{cases}
\left(\sqrt{\frac{d}{\ell}} + \frac{n\delta}{\ell}\right)\|\theta\| +\sqrt{\frac{d}{n}}
+ \frac{1}{\|\theta\|}\left(\sqrt{\frac{1}{k}}\sqrt{\frac{d}{n}} + \frac{d}{n}\right) , & \|\theta\|^2 \leq 1\\
\sqrt{\frac{d}{n}} + \frac{d}{n}, & \|\theta\|^2 \geq 1 ,
\end{cases}
\end{equation}
where we take $\ell = n$ in the case $\|\theta\|^2 \geq 1$. 
\end{thm}
As expected the additional cost of not knowing $g(\delta)$ is of order $n\delta/\ell$. Next we propose the optimal number of buckets $\ell=\ell^{**}$ for this procedure. It is given by 
\begin{equation}\label{ell **}
    \ell^{**}=d\vee\left\lceil \left(n\delta\vee\frac{\delta^{2/3}n^{4/3}\|\theta\|^{4/3}}{d^{1/3}}\vee n \|\theta\|^2\right) \right\rceil\wedge n.
\end{equation}
The corresponding estimation rates take different values according to several scenarios. 
\begin{prop}\label{prop: rates for theta_2} 
Assume that $d\leq n$. With probability greater than $1-c_1\cdot e^{-c_0d}$, the estimator $\tilde{\theta}(\ell^{**})$, where $\ell^{**}$ is defined in \eqref{ell **}, achieves the following: \begin{itemize}
    \item When $\delta\leq\sqrt{\frac{d}{n}}$:
\[
\ell^2(\tilde{\theta}(\ell^{**}),\theta)\lesssim
\begin{cases}
\sqrt{\frac{\delta d}{n}} + \frac{d}{n}, & \|\theta\|^2 \leq \sqrt{\frac{\delta d}{n}}\vee \frac{d}{n}\\
\left(\frac{\frac{\delta d}{n}}{\|\theta\|}\right)^{\frac{2}{3}}, & \sqrt{\frac{\delta d}{n}}\vee \frac{d}{n} \leq \|\theta\|^2 \leq \frac{\delta^2n}{d}\vee \frac{d}{n}\\
\frac{d}{n}, & \frac{\delta^2n}{d}\vee \frac{d}{n} \leq \|\theta\|^2. \\
\end{cases}
\label{eqn: rate without delta}
\]
\item When $\sqrt{\frac{d}{n}}\leq\delta$:
\[
\ell^2(\tilde{\theta}(\ell^{**}),\theta)\lesssim
\begin{cases}
\sqrt{\frac{\delta d}{n}} +\frac{d}{n}, & \|\theta\|^2 \leq \sqrt{\frac{\delta d}{n}} \\
\left(\frac{\frac{\delta d}{n}}{\|\theta\|}\right)^{\frac{2}{3}} , & \sqrt{\frac{\delta d}{n}} \leq \|\theta\|^2 \leq \frac{\sqrt{\frac{d}{n}}}{\delta}\\
\frac{d}{n}+\frac{\frac{d}{n}}{\|\theta\|^2}, & \frac{\sqrt{\frac{d}{n}}}{\delta} \leq \|\theta\|^2 .
\end{cases}
\]
\end{itemize}
Moreover, it remains true that $\ell^2(\tilde{\theta}(\ell^{**}),\theta)\lesssim\sqrt{\frac{\delta d}{n}} + \frac{d}{n}$, with probability greater than $1-c_1 e^{-c_0d}$.
\end{prop}

 \begin{figure}[!ht]
    \begin{subfigure}[t]{.5\textwidth}
     \begin{tikzpicture}

\draw[->] (0,0) -- (5,0) node[anchor=north] {$\|\theta\|^2$};
\draw	(0,0) node[anchor=north] {0}
		(1,0) node[anchor=north] {$\sqrt{\frac{\delta d}{n}}$}
            (2,-0.15) node[anchor=north] {$\delta$}
            (4,-0.15) node[anchor=north] {$1$}
            (2.5,0) node[anchor=north] {$\sqrt{\frac{d}{n}}$}
            (2,4.1) node[anchor=north] {$\left(\frac{\frac{\delta d}{n}}{\|\theta\|}\right)^{\frac{2}{3}}$}
            (1.1,2) node[anchor=north]{$\frac{\delta d}{n\|\theta\|^2}$}
            (3.2,5) node[anchor=north] {$\frac{ d}{n\|\theta\|^2}$}
		(3.5,0) node[anchor=north] {$\frac{\delta^2 n}{d}$};
\draw	(-0.5, 4) node{ $\sqrt{\frac{\delta d}{n}}$}
(-0.5, 6) node{ $\sqrt{\frac{ d}{n}}$}
		(-0.2, 1) node{ $\frac{d}{n}$ };

\draw[->] (0,0) -- (0,7); 
\draw[dotted] (1,0) -- (1,4);
\draw[dotted] (3.5,0) -- (3.5,1);
\draw[dotted] (2,0) -- (2,1);
\draw[dotted] (2.5,0) -- (2.5,6);
\draw[dotted] (0,1) -- (3,1);
\draw[dotted] (4,0) -- (4,1);

\draw[thick,dashdotted] (0,4) -- (1,4);
\draw[thick] (0,6) -- (2.5,6);
\draw[thick,dashed](2,1) -- (3,1);
\draw[thick,dashdotted] (3,1) -- (5,1);
\draw[thick] (4,1) -- (5,1);

\draw[thick,dotted]  (1,4) parabola[bend at end] (3.5,1); 
\draw[thick,dashed] (1,4) parabola[bend at end] (2,1); 
\draw[thick,solid]  (2.5,6) parabola[bend at end] (4,1); 


\end{tikzpicture}
        \caption{Case $\delta \leq \sqrt{d/n}$}
    \end{subfigure}%
    \begin{subfigure}[t]{.5\textwidth}
 \begin{tikzpicture}

\draw[->] (0,0) -- (5,0) node[anchor=north] {$\|\theta\|^2$};
\draw	(0,0) node[anchor=north] {0}
		(0.8,0) node[anchor=north] {$\sqrt{\frac{\delta d}{n}}$}
  (1.5,0) node[anchor=north] {$\sqrt{\frac{d}{n}}$}
            (2.25,-0.15) node[anchor=north] {$\delta$}
		(3,0) node[anchor=north] {$\frac{\sqrt{d/n}}{\delta}$}
            (2,4) node[anchor=north] {$\left(\frac{\frac{\delta d}{n}}{\|\theta\|}\right)^{\frac{2}{3}}$}
            (3.2,3) node[anchor=north] {$\frac{d}{n\|\theta\|^2}$}
            (1.2,2) node[anchor=north]{$\frac{\delta d}{n \|\theta\|^2}$}
            (4,-0.15) node[anchor=north] {$1$};
\draw	(-0.5, 4) node{ $\sqrt{\frac{\delta d}{n}}$}
(-0.5, 6) node{ $\sqrt{\frac{ d}{n}}$}
		(-0.2, 1) node{ $\frac{d}{n}$ };

\draw[->] (0,0) -- (0,7);
\draw[dotted] (1,0) -- (1,4);
\draw[dotted] (2.9,0) -- (2.9,2);
\draw[dotted] (2.25,0) -- (2.25,1);
\draw[dotted] (4,0) -- (4,1);
\draw[dotted] (0,1) -- (3,1);
\draw[dotted] (1.5,0) -- (1.5,6);

\draw[thick,dashdotted] (0,4) -- (1,4);
\draw[thick] (0,6) -- (1.5,6);
\draw[thick,dashed](2.25,1) -- (4,1);
\draw[thick] (4,1) -- (5,1);

\draw[thick,dotted]  (1,4) parabola[bend at end] (2.9,2); 
\draw[thick,solid]  (1.5,6) parabola[bend at end] (4,1); 
\draw[thick,dashed] (1,4) parabola[bend at end] (2.25,1); 


\matrix [draw,below left] at (current bounding box.north east) {
  \node [dashed,label=right:$\hat{\theta}(\ell^*)$] {- -}; \\
  \node [solid,label=right:$\tilde{\theta}(\ell^{**})$] {....}; \\
    \node [solid,label=right:$\hat{\theta}(n)$] {---}; \\
};
\end{tikzpicture}
        \caption{Case $\delta \geq \sqrt{d/n}$}
    \end{subfigure}%
 \caption{Upper bounds of the risk of estimators $\hat\theta(\ell^{*})$ (Proposition \ref{prop: theoretical rates}), $\tilde\theta(\ell^{**})$ (Proposition \ref{prop: rates for theta_2}) and the vanilla spectral method $\hat{\theta}(n)$.}
  \label{fig:tikz}
\end{figure}
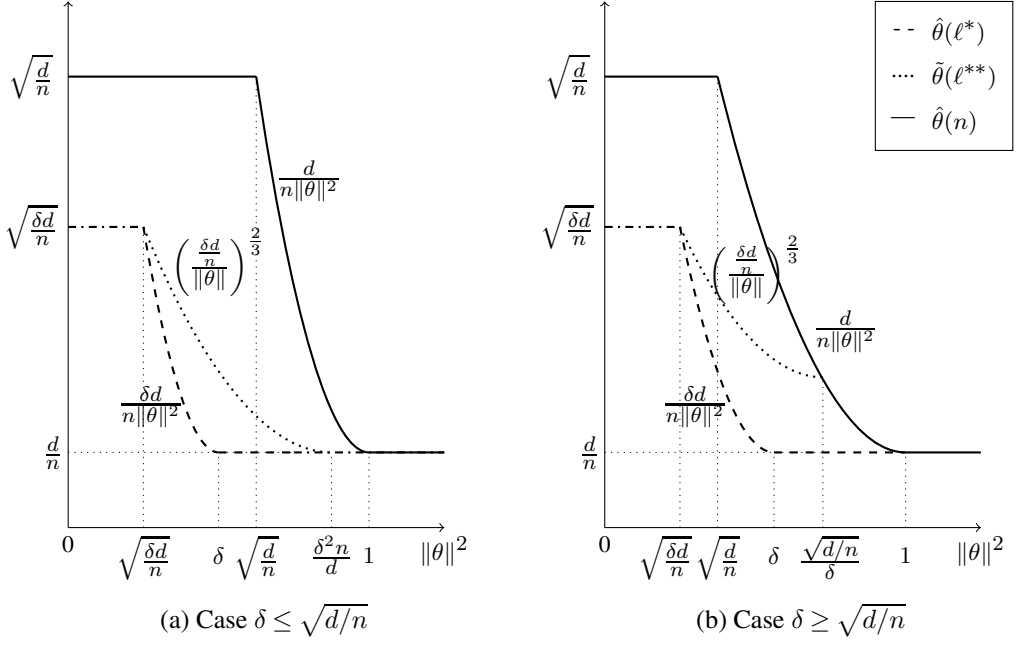
The  rates of Proposition \ref{prop: rates for theta_2} are displayed in Figure \ref{fig:tikz}. Based on our proof we can write the estimation error in Proposition \ref{prop: rates for theta_2} in a more compact form given by
$$
\ell^2(\tilde{\theta}(\ell^{**}),\theta) \lesssim \left(\sqrt{\frac{\delta d }{n}} + \frac{d}{n}\right) \wedge\left(\frac{\ell^{**} d}{n\|\theta\|^2} + \frac{d}{n}\right).$$
As a result $\tilde{\theta}(\ell^{**})$ is not locally minimax optimal but remains globally minimax optimal. Observe that it is always better than the vanilla spectral algorithm (without bucketing) even locally. Not knowing $g(\delta)$ leads to sub-optimal rates locally which suggests that full adaptation might not be possible for locally minimax optimal procedures. That being said, if the goal is to achieve global minimax optimality then there is still hope.
 With this in mind the next technical lemma sets a basis for constructing an adaptive algorithm without the knowledge of optimal $\ell^{**}$. This is done by proposing a surrogate upper bound that does not depend on $\delta$.
\begin{lem}
    For any number of buckets $\ell\geq\ell^{**}$ it holds that
    \begin{equation}\label{eq:surroagte}
        \ell^2(\tilde{\theta}(\ell),\theta)\lesssim\frac{\sqrt{d\ell}}{n}\wedge\left(\frac{\ell d}{n^2\|\theta \|^2 } + \frac{d}{n}\right) ,
    \end{equation}
    with probability  $1-c_1\cdot e^{-c_0d}$ .
    \label{lem: property for large ell}
\end{lem}

While choosing $\ell=\ell^{**}$ results in a small estimation error (even locally), we get higher global rates for $\ell\geq \ell^{**}$. Indeed $\sqrt{\frac{\delta d}{n}}\lesssim\frac{\sqrt{d\ell}}{n}$, under the assumption $\ell\geq \ell^{**}$. This observation shall help rule out values of $\ell$ for which the estimation rate is too large. In what follows, we focus on the case $\|\theta\|^2\leq 1$ and propose an adaptive procedure which is inspired from the well-known Lepski's method, widely used for adaptation. The other case $\|\theta\|^2\geq 1$ is easy since we can simply use the spectral method in that case $(\ell = n)$ and get that 
$$
\ell^2(\tilde{\theta}(\tilde{\ell}),\theta) \lesssim \frac{d}{n},
$$
with probability greater than $1-c_1\cdot e^{-c_0d}$. The upper bound \eqref{eq:surroagte} will be useful for adaptation since it does not depend on $\delta$ but it still depends on $\|\theta\|$. We proceed in two steps. First construct a fully adaptive estimator that is globally minimax optimal. Then,  plug-in the norm of the latter estimator replacing $\|\theta\|$ in \eqref{eq:surroagte} in order to get better rates locally. 

Let us define a grid-line for values of $\ell$ in the set $\{d,d+1,\ldots,n\}$. It is given by  
$$G\dfn\{g_i=2^i d \mid i=0,1,\ldots,M\},$$ where $M$ is the greatest possible integer that $2^M d\leq n$. We define $\tilde{m}$ and $\tilde{\ell}$ in the following way:
\begin{equation*}
    \tilde{m}\dfn \min\biggl \{m\in \{0,1,\ldots,M\}: \;\; \ell^2(\tilde{\theta}(g_k),\tilde{\theta}(g_{k-1}))\leq 4 C^2 \left(\frac{\sqrt{d g_k}}{n}\wedge\frac{d g_k}{n^2\|\tilde\theta(g_k) \|^2 }\right) \;\;\; \text{for all}\; k\geq m+1\biggl\},
\end{equation*}
for some $C$ large enough. We then consider the adaptive number of buckets
\begin{equation}\label{tilde ell}
    \tilde{\ell}\dfn g_{\tilde{m}} = 2^{\tilde{m}}d.
\end{equation}
We are now ready to state the first adaptation result of this section.
\begin{thm}{\label{thm: lepski}}
    Let $\tilde{\ell}$ defined in \eqref{tilde ell}, $\|\theta\|^2 \leq 1$ and $d \leq n$. Then, the fully adaptive estimator $\tilde{\theta}(\tilde{\ell})$  achieves the following rate
$$
\ell^2(\tilde{\theta}(\tilde{\ell}),\theta) \lesssim \left(\sqrt{\frac{\delta d }{n}} + \frac{d}{n}\right),
$$
    with probability $1-c_1\log(n/d)\cdot e^{-c_0 d}$. 
\end{thm}
Theorem \ref{thm: lepski} leads to a fully adaptive estimator that is globally minimax optimal. While it is most likely impossible to achieve local minimax optimal rates adaptively  we can still hope to achieve better rates locally. This is done in a more refined construction that we describe as follows. We define $\tilde{m}_2$ and $\tilde{\ell}_2$ in the following way:
\begin{equation*}
    \tilde{m}_2\dfn \min\biggl \{m\in \{0,1,\ldots,M\}: \;\; \ell^2(\tilde{\theta}(g_k),\tilde{\theta}(g_{k-1}))\leq 4 C^2 \left(\frac{\sqrt{d g_k}}{n}\wedge\frac{d g_k}{n^2\|\tilde\theta(\tilde{\ell}) \|^2 }\right) \;\;\; \text{for all}\; k\geq m+1\biggl\},
\end{equation*}
for some $C$ large enough. We then consider the adaptive number of buckets
\begin{equation}\label{tilde ell 2}
    \tilde{\ell}_2\dfn g_{\tilde{m}_2} = 2^{\tilde{m}_2}d.
\end{equation}
Observe that, in the definition of $\tilde{m}_2$, the norm of $\|\theta\|$ is replaced by $\|\tilde{\theta}(\tilde{\ell})\|$ and not by $\|\tilde{\theta}(g_k)\|$ as we did in the definition of $\tilde{m}$. We are now ready to state the main result of this section.

\begin{thm}{\label{thm: lepski2}}
    Let $\tilde{\ell}_2$ defined in \eqref{tilde ell 2}, $\|\theta\|^2 \leq 1$ and $d \leq n$. Then, the fully adaptive estimator $\tilde{\theta}(\tilde{\ell}_2)$ achieves the following rate
$$
\ell^2(\tilde{\theta}(\tilde{\ell}_2),\theta) \lesssim \left(\sqrt{\frac{\delta d }{n}} + \frac{d}{n}\right)\wedge\left(\frac{\ell^{**} d}{n^2\|\theta\|^2} + \frac{d}{n}\right),
$$
    with probability $1-c_1\log(n/d)\cdot e^{-c_0 d}$. In particular, we also have, globally,  that 
    $$
    \ell^2(\tilde{\theta}(\tilde{\ell}_2),\theta) \lesssim \sqrt{\frac{\delta d}{n}} + \frac{d}{n},
    $$
     with probability $1-c_1\log(n/d)\cdot e^{-c_0 d}$.
\end{thm}
As a result, we have showed that full adaptation is possible for globally minimax optimal procedures (up to logarithmic factors). One may also observe that the bound in Theorem \ref{thm: lepski2} is exactly the same as the one in Proposition \ref{prop: rates for theta_2}  replacing $\ell^{**}$ by its value. Hence we recover the local estimation rates of $\tilde{\theta}(\ell^{**})$ (cf. Figure \ref{fig:tikz}) in a fully adaptive fashion. In other terms, our adaptive procedure $\tilde{\theta}(\tilde{\ell}_2)$ is strictly better than the vanilla spectral method not only globally but also locally. 
\section{Conclusion and future work}
\subsection{Conclusion}
As a conclusion, our contribution represents an improvement of previous results when $\delta$ is known. Not only we get rid of the logarithmic factors, but we show that our approach achieves optimal rates locally under tighter bounds and holds under (sub)Gaussian noise.

 More specifically, we show that the optimal procedure depends on both $\delta$ and $\|\theta\|$. We then propose a semi-adaptive approach (adaptive to the signal), showing that minimax optimality requires only knowledge of $\delta$. As for full adaptation, we claim that plug-in estimation of $\delta$ might not be optimal and that local minimax optimality is rather difficult  without prior knowledge of $\delta$. We finally propose an adaptive procedure that bypasses estimation of $\delta$ and that is globally minimax optimal.

 In summary, our study refines center estimation in the high-dimensional Markovian binary GMM, and shows that we can reach the improved global rate $\sqrt{\delta d/n} + d/n$ instead of the classical rate $\sqrt{ d/n} + d/n$ and that it can be done adaptively. Hence the hidden Markovian structure of the labels can be leveraged to improve center estimation without further information about this structure. 

\subsection{Future work}
As for future research, one interesting direction is to characterize the adaptive locally minimax optimal rates where we conjecture the rates to be different from the non adaptive ones in general. An intriguing phenomenon observed during our study is that, when $d \leq n$, it seems possible to estimate the center of the mixture while impossible to estimate the parameter $\delta$. As opposed to our intuition, the $d$-dimensional vector $\theta$ seems easier to estimate than the scalar $\delta$. A very interesting question is to characterize the minimax rate of estimation for $\delta$ in the high dimensional mixture model. The challenging part, would be to capture the right dimension dependence in the lower bound. We believe that this is an interesting direction of research.


Another extension of our work could be a generalization of the current binary Markov chain into more sophisticated graph structures. In the current model, the assumption is that each new observation depends only on the previous one. We can extend the memory by assuming that each signs depends on two or more previous data points as an example.

\bibliographystyle{imsart-nameyear}

\begin{thebibliography}{23}

\bibitem[\protect\citeauthoryear{Abbe, Fan and Wang}{2022}]{abbe2022}
\begin{barticle}[author]
\bauthor{\bsnm{Abbe},~\bfnm{Emmanuel}\binits{E.}},
  \bauthor{\bsnm{Fan},~\bfnm{Jianqing}\binits{J.}} \AND
  \bauthor{\bsnm{Wang},~\bfnm{Kaizheng}\binits{K.}}
(\byear{2022}).
\btitle{An $\ell_p$ theory of PCA and spectral clustering}.
\bjournal{The Annals of Statistics}
\bvolume{50}
\bpages{2359--2385}.
\end{barticle}
\endbibitem

\bibitem[\protect\citeauthoryear{Abraham, Gassiat and Naulet}{2023}]{ab_2023}
\begin{barticle}[author]
\bauthor{\bsnm{Abraham},~\bfnm{Kweku}\binits{K.}},
  \bauthor{\bsnm{Gassiat},~\bfnm{Elisabeth}\binits{E.}} \AND
  \bauthor{\bsnm{Naulet},~\bfnm{Zacharie}\binits{Z.}}
(\byear{2023}).
\btitle{Fundamental Limits for Learning Hidden Markov Model Parameters}.
\bjournal{IEEE Transactions on Information Theory}
\bvolume{69}
\bpages{1777–1794}.
\bdoi{10.1109/tit.2022.3213429}
\end{barticle}
\endbibitem

\bibitem[\protect\citeauthoryear{Alexandrovich, Holzmann and
  Leister}{2015}]{alex2015}
\begin{bmisc}[author]
\bauthor{\bsnm{Alexandrovich},~\bfnm{Grigory}\binits{G.}},
  \bauthor{\bsnm{Holzmann},~\bfnm{Hajo}\binits{H.}} \AND
  \bauthor{\bsnm{Leister},~\bfnm{Anna}\binits{A.}}
(\byear{2015}).
\btitle{Nonparametric identification and maximum likelihood estimation for
  hidden Markov model}.
\end{bmisc}
\endbibitem

\bibitem[\protect\citeauthoryear{Balakrishnan, Wainwright and Yu}{2017}]{balak}
\begin{barticle}[author]
\bauthor{\bsnm{Balakrishnan},~\bfnm{Sivaraman}\binits{S.}},
  \bauthor{\bsnm{Wainwright},~\bfnm{Martin~J}\binits{M.~J.}} \AND
  \bauthor{\bsnm{Yu},~\bfnm{Bin}\binits{B.}}
(\byear{2017}).
\btitle{Statistical guarantees for the EM algorithm: From population to
  sample-based analysis}.
\end{barticle}
\endbibitem

\bibitem[\protect\citeauthoryear{Castro, Élisabeth Gassiat and
  Lacour}{2015}]{decastro}
\begin{bmisc}[author]
\bauthor{\bsnm{Castro},~\bfnm{Yohann~De}\binits{Y.~D.}},
  \bauthor{\bparticle{Élisabeth} \bsnm{Gassiat}} \AND
  \bauthor{\bsnm{Lacour},~\bfnm{Claire}\binits{C.}}
(\byear{2015}).
\btitle{Minimax adaptive estimation of nonparametric hidden Markov models}.
\end{bmisc}
\endbibitem

\bibitem[\protect\citeauthoryear{Chen and Yang}{2021}]{chen2021cutoff}
\begin{barticle}[author]
\bauthor{\bsnm{Chen},~\bfnm{Xiaohui}\binits{X.}} \AND
  \bauthor{\bsnm{Yang},~\bfnm{Yun}\binits{Y.}}
(\byear{2021}).
\btitle{Cutoff for exact recovery of gaussian mixture models}.
\bjournal{IEEE Transactions on Information Theory}
\bvolume{67}
\bpages{4223--4238}.
\end{barticle}
\endbibitem

\bibitem[\protect\citeauthoryear{Dwivedi et~al.}{2018}]{raz2}
\begin{barticle}[author]
\bauthor{\bsnm{Dwivedi},~\bfnm{Raaz}\binits{R.}},
  \bauthor{\bsnm{Khamaru},~\bfnm{Koulik}\binits{K.}},
  \bauthor{\bsnm{Wainwright},~\bfnm{Martin~J}\binits{M.~J.}},
  \bauthor{\bsnm{Jordan},~\bfnm{Michael~I}\binits{M.~I.}} \betal{et~al.}
(\byear{2018}).
\btitle{Theoretical guarantees for EM under misspecified Gaussian mixture
  models}.
\bjournal{Advances in Neural Information Processing Systems}
\bvolume{31}.
\end{barticle}
\endbibitem

\bibitem[\protect\citeauthoryear{Dwivedi et~al.}{2019}]{raz1}
\begin{barticle}[author]
\bauthor{\bsnm{Dwivedi},~\bfnm{Raaz}\binits{R.}},
  \bauthor{\bsnm{Ho},~\bfnm{Nhat}\binits{N.}},
  \bauthor{\bsnm{Khamaru},~\bfnm{Koulik}\binits{K.}},
  \bauthor{\bsnm{Wainwright},~\bfnm{Martin~J.}\binits{M.~J.}},
  \bauthor{\bsnm{Jordan},~\bfnm{Michael~I.}\binits{M.~I.}} \AND
  \bauthor{\bsnm{Yu},~\bfnm{Bin}\binits{B.}}
(\byear{2019}).
\btitle{Challenges with EM in application to weakly identifiable mixture
  models}.
\bjournal{ArXiv}
\bvolume{abs/1902.00194}.
\end{barticle}
\endbibitem

\bibitem[\protect\citeauthoryear{Dwivedi et~al.}{2020a}]{raz3}
\begin{barticle}[author]
\bauthor{\bsnm{Dwivedi},~\bfnm{Raaz}\binits{R.}},
  \bauthor{\bsnm{Ho},~\bfnm{Nhat}\binits{N.}},
  \bauthor{\bsnm{Khamaru},~\bfnm{Koulik}\binits{K.}},
  \bauthor{\bsnm{Wainwright},~\bfnm{Martin}\binits{M.}},
  \bauthor{\bsnm{Jordan},~\bfnm{Michael}\binits{M.}} \AND
  \bauthor{\bsnm{Yu},~\bfnm{Bin}\binits{B.}}
(\byear{2020}a).
\btitle{Sharp analysis of expectation-maximization for weakly identifiable
  models}.
\bpages{1866--1876}.
\end{barticle}
\endbibitem

\bibitem[\protect\citeauthoryear{Dwivedi et~al.}{2020b}]{raz4}
\begin{bmisc}[author]
\bauthor{\bsnm{Dwivedi},~\bfnm{Raaz}\binits{R.}},
  \bauthor{\bsnm{Ho},~\bfnm{Nhat}\binits{N.}},
  \bauthor{\bsnm{Khamaru},~\bfnm{Koulik}\binits{K.}},
  \bauthor{\bsnm{Jordan},~\bfnm{Michael~I.}\binits{M.~I.}},
  \bauthor{\bsnm{Wainwright},~\bfnm{Martin~J.}\binits{M.~J.}} \AND
  \bauthor{\bsnm{Yu},~\bfnm{Bin}\binits{B.}}
(\byear{2020}b).
\btitle{Singularity, Misspecification, and the Convergence Rate of EM}.
\end{bmisc}
\endbibitem

\bibitem[\protect\citeauthoryear{Gassiat, Kaddouri and
  Naulet}{2023}]{gassiat2023model}
\begin{barticle}[author]
\bauthor{\bsnm{Gassiat},~\bfnm{Elisabeth}\binits{E.}},
  \bauthor{\bsnm{Kaddouri},~\bfnm{Ibrahim}\binits{I.}} \AND
  \bauthor{\bsnm{Naulet},~\bfnm{Zacharie}\binits{Z.}}
(\byear{2023}).
\btitle{Model-based Clustering using Non-parametric Hidden Markov Models}.
\bjournal{arXiv preprint arXiv:2309.12238}.
\end{barticle}
\endbibitem

\bibitem[\protect\citeauthoryear{Giraud and Verzelen}{2019}]{giraud2019partial}
\begin{barticle}[author]
\bauthor{\bsnm{Giraud},~\bfnm{Christophe}\binits{C.}} \AND
  \bauthor{\bsnm{Verzelen},~\bfnm{Nicolas}\binits{N.}}
(\byear{2019}).
\btitle{Partial recovery bounds for clustering with the relaxed $ K $-means}.
\bjournal{Mathematical Statistics and Learning}
\bvolume{1}
\bpages{317--374}.
\end{barticle}
\endbibitem

\bibitem[\protect\citeauthoryear{Klusowski and Brinda}{2016}]{EM-klus}
\begin{barticle}[author]
\bauthor{\bsnm{Klusowski},~\bfnm{Jason~M}\binits{J.~M.}} \AND
  \bauthor{\bsnm{Brinda},~\bfnm{WD}\binits{W.}}
(\byear{2016}).
\btitle{Statistical guarantees for estimating the centers of a two-component
  Gaussian mixture by EM}.
\bjournal{arXiv preprint arXiv:1608.02280}.
\end{barticle}
\endbibitem

\bibitem[\protect\citeauthoryear{Lehéricy}{2021}]{leh2021}
\begin{bmisc}[author]
\bauthor{\bsnm{Lehéricy},~\bfnm{Luc}\binits{L.}}
(\byear{2021}).
\btitle{Nonasymptotic control of the MLE for misspecified nonparametric hidden
  Markov models}.
\end{bmisc}
\endbibitem

\bibitem[\protect\citeauthoryear{L{\"o}ffler, Zhang and
  Zhou}{2021}]{loffler2021optimality}
\begin{barticle}[author]
\bauthor{\bsnm{L{\"o}ffler},~\bfnm{Matthias}\binits{M.}},
  \bauthor{\bsnm{Zhang},~\bfnm{Anderson~Y}\binits{A.~Y.}} \AND
  \bauthor{\bsnm{Zhou},~\bfnm{Harrison~H}\binits{H.~H.}}
(\byear{2021}).
\btitle{Optimality of spectral clustering in the Gaussian mixture model}.
\bjournal{The Annals of Statistics}
\bvolume{49}
\bpages{2506--2530}.
\end{barticle}
\endbibitem

\bibitem[\protect\citeauthoryear{Lu and Zhou}{2016}]{Lu}
\begin{barticle}[author]
\bauthor{\bsnm{Lu},~\bfnm{Yu}\binits{Y.}} \AND
  \bauthor{\bsnm{Zhou},~\bfnm{Harrison~H}\binits{H.~H.}}
(\byear{2016}).
\btitle{Statistical and computational guarantees of lloyd's algorithm and its
  variants}.
\bjournal{arXiv preprint arXiv:1612.02099}.
\end{barticle}
\endbibitem

\bibitem[\protect\citeauthoryear{Ndaoud}{2019}]{ndaoud2019interplay}
\begin{binproceedings}[author]
\bauthor{\bsnm{Ndaoud},~\bfnm{Mohamed}\binits{M.}}
(\byear{2019}).
\btitle{Interplay of minimax estimation and minimax support recovery under
  sparsity}.
In \bbooktitle{Algorithmic Learning Theory}
\bpages{647--668}.
\bpublisher{PMLR}.
\end{binproceedings}
\endbibitem

\bibitem[\protect\citeauthoryear{Ndaoud}{2020}]{ndaoud2020scaled}
\begin{barticle}[author]
\bauthor{\bsnm{Ndaoud},~\bfnm{Mohamed}\binits{M.}}
(\byear{2020}).
\btitle{Scaled minimax optimality in high-dimensional linear regression: A
  non-convex algorithmic regularization approach}.
\bjournal{arXiv preprint arXiv:2008.12236}.
\end{barticle}
\endbibitem

\bibitem[\protect\citeauthoryear{Ndaoud}{2022}]{ndaoud2022sharp}
\begin{barticle}[author]
\bauthor{\bsnm{Ndaoud},~\bfnm{Mohamed}\binits{M.}}
(\byear{2022}).
\btitle{Sharp optimal recovery in the two component gaussian mixture model}.
\bjournal{The Annals of Statistics}
\bvolume{50}
\bpages{2096--2126}.
\end{barticle}
\endbibitem

\bibitem[\protect\citeauthoryear{Rudelson and
  Vershynin}{2013}]{rudelson2013hanson}
\begin{barticle}[author]
\bauthor{\bsnm{Rudelson},~\bfnm{Mark}\binits{M.}} \AND
  \bauthor{\bsnm{Vershynin},~\bfnm{Roman}\binits{R.}}
(\byear{2013}).
\btitle{Hanson-Wright inequality and sub-Gaussian concentration}.
\end{barticle}
\endbibitem

\bibitem[\protect\citeauthoryear{Vershynin}{2010}]{vershynin2010introduction}
\begin{barticle}[author]
\bauthor{\bsnm{Vershynin},~\bfnm{Roman}\binits{R.}}
(\byear{2010}).
\btitle{Introduction to the non-asymptotic analysis of random matrices}.
\bjournal{arXiv preprint arXiv:1011.3027}.
\end{barticle}
\endbibitem

\bibitem[\protect\citeauthoryear{Wu and Zhou}{2021}]{EM-wu}
\begin{barticle}[author]
\bauthor{\bsnm{Wu},~\bfnm{Yihong}\binits{Y.}} \AND
  \bauthor{\bsnm{Zhou},~\bfnm{Harrison~H}\binits{H.~H.}}
(\byear{2021}).
\btitle{Randomly initialized EM algorithm for two-component Gaussian mixture
  achieves near optimality in $O(\sqrt{n})$ iterations}.
\bjournal{Mathematical Statistics and Learning}
\bvolume{4}.
\end{barticle}
\endbibitem

\bibitem[\protect\citeauthoryear{Zhang and Weinberger}{2022}]{yihan}
\begin{barticle}[author]
\bauthor{\bsnm{Zhang},~\bfnm{Yihan}\binits{Y.}} \AND
  \bauthor{\bsnm{Weinberger},~\bfnm{Nir}\binits{N.}}
(\byear{2022}).
\btitle{Mean Estimation in High-Dimensional Binary Markov Gaussian Mixture
  Models}.
\bjournal{Advances in Neural Information Processing Systems}
\bvolume{35}
\bpages{19673--19686}.
\end{barticle}
\endbibitem

\end{thebibliography}

\appendix
\section{Proofs of main results}

\subsection{Proof of Theorem \ref{thm: known delta loss bound}}
Let us consider the following estimator
\[\label{def: hat_theta}\hat{\theta}(\ell) \coloneqq \sqrt{\frac{\ell}{\mathbf{E}\|\bar{\eta}\|^2}\left(\lambda_{max}(\hat{\Sigma})-\frac{1}{k}\right)_+}  v_{max}(\hat{\Sigma}).\]
Sometimes when the context is clear we will just write $\hat{\theta}$ instead of $\hat{\theta}(\ell)$.
We remind the reader that $\Sigma=\frac{\mathbf{E}\|\bar{\eta}\|^2}{\ell}\theta\theta^T+\frac{\mathbf{I}_d}{k}$, it becomes clear that eigenvalues of $\Sigma$ are $\lambda_{max}(\Sigma)=\|\theta\|^2\frac{\mathbf{E}\|\bar{\eta}\|^2}{\ell}+\frac{1}{k}$ and $\lambda_2=...=\lambda_d=\frac{1}{k}$.

We begin the analysis of our estimator by Weyl's inequality that states that
\[\left|\lambda_{max}(\hat{\Sigma})-\lambda_{max}(\Sigma)\right| \leq \|\hat{\Sigma}-\Sigma\|_{op}.\]
In particular this leads to
$$
\|\theta\|^2\frac{\mathbf{E}\|\bar{\eta}\|^2}{\ell} \leq \lambda_{max}(\hat{\Sigma}) - \frac{1}{k} + \|\hat{\Sigma}-\Sigma\|_{op}.
$$
It comes out that if $\|\hat{\theta}\|=0$ then 
\begin{equation}\label{eq:thm1:1}
|\|\hat{\theta}\|-\|\theta\| |\leq \sqrt{\frac{\ell}{\mathbf{E}\|\bar{\eta}\|^2}}\frac{\|\hat{\Sigma}-\Sigma\|_{op}}{\|\theta\|} .
\end{equation}
Moreover we also have for $\|\hat{\theta}\|\neq 0$ that
\begin{align*}
    \left|\|\hat{\theta}\|-\|\theta\|\right|&=\frac{\left|\|\hat{\theta}\|^2-\|\theta\|^2 \right|}{\|\hat{\theta}\|+\|\theta\|} \\&\leq \frac{\left|\|\hat{\theta}\|^2-\|\theta\|^2 \right|}{\|\theta\|} \\&= \frac{\ell}{\mathbf{E}\|\bar\eta\|^2}\frac{\left|\lambda_{max}(\hat{\Sigma})-\lambda_{max}(\Sigma)\right|}{\|\theta\|} \\&\leq \frac{\ell}{\mathbf{E}\|\bar\eta\|^2}\frac{\|\hat{\Sigma}-\Sigma\|_{op}}{\|\theta\|}.\label{eqn: norm trick}
\end{align*}
We conclude that we always have 
\begin{equation}\label{eq:weyl:conclusion}
   |\|\hat{\theta}\|-\|\theta\| | \leq \frac{\ell}{\mathbf{E}\|\bar\eta\|^2}\frac{\|\hat{\Sigma}-\Sigma\|_{op}}{\|\theta\|},
\end{equation}
since $\mathbf{E}\|\bar\eta\|^2 \leq \ell$. Next observe that the largest eigengap of $\Sigma$ is $\frac{\mathbf{E}\|\bar{\eta}\|^2}{\ell}\|\theta\|^2$. Hence, and using the Davis-Kahan's perturbation bound, we get that 
\begin{equation}\label{eqn: eigenvector loss bound}
\ell(\hat{v}_{max},v_{max})\leq \frac{4\ell}{\mathbf{E} \|\bar{\eta}\|^2}\frac{\|\hat\Sigma-\Sigma\|_{op}}{ \|\theta\|^2}.
\end{equation}
Combining \eqref{eq:weyl:conclusion} and \eqref{eqn: eigenvector loss bound}we finally get that
\[
\ell(\hat{\theta},\theta) \leq \left|\|\hat{\theta}\|-\|\theta\|\right| + \|\theta\|\ell(\hat{v}_{max},v_{max}) \leq \frac{5\ell}{\mathbf{E} \|\bar{\eta}\|^2}\frac{\|\hat{\Sigma}-\Sigma\|_{op}}{\|\theta\|}.
\label{eqn: loss bound via operator norm}
\]
Now we continue with the analysis of the difference of covariance and sample covariance matrices. Recall that
\begin{equation}\label{eq:thm:covariance:decomp}
\hat{\Sigma}-\Sigma=\frac{\theta\theta^T(\|\bar{\eta}\|^2-\mathbf{E}{\|\bar{\eta}\|^2})}{\ell} +\frac{w\bar{\eta}\cdot \theta^T+\theta \cdot(w\bar{\eta})^T}{\sqrt{k}\ell}+\frac{1}{k}\left(\frac{ww^T}{\ell}-\mathbf{I}_d\right).
\end{equation}
Observe that for $\xi:= w\bar{\eta}/\|\bar{\eta}\|$ is a $1-$sub-Gaussian vector. Hence $$\left\|\frac{1}{\sqrt{k}\ell}\left((w\bar{\eta})\theta^T+\theta(w\bar{\eta})^T \right)\right\|_{op}\leq \frac{2}{\sqrt{k}\ell} \|\theta\|\|\bar{\eta}\|\|\xi\|. $$
As a result
\[
\|\hat{\Sigma}-\Sigma\|_{op}\leq \frac{\|\theta\|^2|\|\bar{\eta}\|^2-\mathbf{E}{\|\bar{\eta}\|^2}|}{\ell} + \frac{2}{\sqrt{k}\ell}\|\theta\|\|\bar{\eta}\|\|\xi\|+\frac{1}{k}\left\|\frac{ww^T}{\ell}-I_d\right\|_{op}.
\]
using Lemma \ref{lem:concentration of three terms} we conclude that
\[
\|\hat{\Sigma}-\Sigma\|_{op}\lesssim \sqrt{\frac{d}{\ell}}\|\theta\|^2 + \frac{\sqrt{\ell d}}{\sqrt{k}\ell}\|\theta\| + \frac{d}{k\ell}+\frac{1}{k}\sqrt{\frac{d}{\ell}} \lesssim
\sqrt{\frac{d}{\ell}}\|\theta\|^2 + \sqrt{\frac{d}{n}}\|\theta\| +\left(\frac{d}{n}+\frac{1}{k}\sqrt{\frac{d}{\ell}}\right).
\label{eqn: final operator concentration}
\]
Combining \eqref{eqn: final operator concentration} and \eqref{eqn: loss bound via operator norm} we end up with the desired loss bound
\[
\ell(\hat{\theta},\theta) \lesssim \sqrt{\frac{d}{\ell}}\|\theta\| +\sqrt{\frac{d}{n}}
+ \frac{1}{\|\theta\|}\left(\sqrt{\frac{d}{k n}}+ \frac{d}{n}\right), \label{eqn: final loss bound}
\]
with probability greater than $1-c_1e^{-c_0d}$. 

For the case $\|\theta\|\geq 1$, and by choosing $\ell = n$ there is no bucketing involved and the procedure boils down to the vanilla spectral method. We repeat the same proof except that now $\mathbf{E}\|\bar \eta\|^2=\ell$ and 
hence the difference of covariance matrices is decomposed as follows
$$\hat{\Sigma}-\Sigma=\frac{(w\bar{\eta})\theta^T+\theta(w\bar{\eta})^T}{n}+\left(\frac{ww^T}{n}-\mathbf{I}_d\right).$$
It comes out that with probability greater than $1-c_1e^{-c_0d}$, we have
$$\|\hat{\Sigma}-\Sigma\|_{op}\lesssim \sqrt{\frac{d}{n}}\|\theta\| +\left(\frac{d}{n}+\sqrt{\frac{d}{n}}\right).$$
Lastly, we get, with probability greater than $1-c_1e^{-c_0d}$, that
$$\ell(\hat{\theta},\theta) \lesssim \frac{\|\hat{\Sigma}-\Sigma\|_{op}}{\|\theta\|}\lesssim\sqrt{\frac{d}{n}} +\frac{1}{\|\theta\|}\left(\frac{d}{n}+\sqrt{\frac{d}{n}}\right)\lesssim \sqrt{\frac{d}{n}}+ \frac{d}{n},$$
since $\|\theta\|\geq 1$.

\subsection{Proof of Proposition \ref{prop: theoretical rates}}
Recall that $\ell^*:= d \vee \lceil n( \delta \vee \|\theta\|^2) \rceil \wedge n$. For the rest of the proof we denote $\hat{\theta}:= \hat{\theta}(\ell^*)$. We begin by observing that for the loss function $\ell(\cdot,\cdot)$ we also have the following decomposition
\begin{align*}
    \ell^2(\hat{\theta},\theta) &\leq 
2\|\hat{\theta}\|^2+2\|\theta\|^2 \\&\leq 
2\left|\|\hat{\theta}\|^2-\|\theta\|^2 \right| + 4\|\theta\|^2 \\& \leq \frac{\ell^*}{2\mathbf{E}\|\bar\eta\|^2}\|\hat{\Sigma}-\Sigma\|_{op}+6\|\theta\|^2 \\& \lesssim 
\sqrt{\frac{d}{n}} \|\theta\| + \left(\frac{d}{n}+\frac{1}{k^*}\sqrt{\frac{d}{\ell^*}}\right) + \sqrt{\frac{d}{\ell^*}}\|\theta\|^2 + \|\theta\|^2,
\end{align*}
with probability greater than $1-c_1e^{-c_0d}$, where we have used the fact that $\mathbf{E}\|\bar\eta\|^2 \geq \ell^*/3$ based on Lemma \ref{lem:expecation}.
Hence as long as $\|\theta\|^2 \leq \sqrt{\delta d/n} \vee d/n$ we also have that $\|\theta\|^2 \leq \delta \vee d/n$ and hence $\ell^* = d\vee \lceil n\delta \rceil $. It comes out that
\[
\ell^2(\hat{\theta},\theta)  \lesssim \sqrt\frac{\delta d}{n}+ \frac{d}{n}.
\label{eq: saved from small values}
\]
Next, in the scenario where $\sqrt{\frac{\delta d}{n}}\vee\frac{d}{n} \leq \|\theta\|^2 \leq \delta \vee \frac{d}{n}$. This case only makes sense for $\delta \geq d/n$. Then, by invoking Theorem \ref{thm: known delta loss bound} with the number of buckets  being $\ell^* = \lceil n\delta \rceil $ we get that
$$\ell(\hat{\theta},\theta) \lesssim \sqrt{\frac{d}{n}}\frac{\|\theta\|}{\sqrt{\delta}} +\sqrt{\frac{d}{n}}
+ \frac{1}{\|\theta\|}\sqrt{\frac{\delta d}{n}} \lesssim \sqrt{\frac{d}{n}}+\frac{\sqrt{\frac{\delta d}{n}}}{\|\theta\|}.$$
For the case $\delta  \vee \frac{d}{n} \leq \|\theta\|^2 \leq 1$, then we have $\ell^*= \lceil n\|\theta\|^2 \rceil$. Using Theorem \ref{thm: known delta loss bound} one more time we get 
$$\ell(\hat{\theta},\theta) \lesssim \sqrt{\frac{d}{n\|\theta\|^2}}\|\theta\| +\sqrt{\frac{d}{n}}
+ \frac{1}{\|\theta\|}\sqrt{\frac{\|\theta\|^2 d}{n}}\lesssim\sqrt{\frac{d}{n}}.$$
Finally, for $\|\theta\|\geq 1$, we get that $\ell^*=n$ and we simply conclude using the corresponding case in Theorem \ref{thm: known delta loss bound}. This proof is now complete.

\subsection{Proof of Lemma \ref{lem: s and theta}}
We consider first the case $\|\theta\|^2 \geq \delta\vee Cd/n$. Using \eqref{eqn: final loss bound} we have
\[
\left| \hat{s}-\|\theta\| \right| \lesssim \sqrt{\frac{d}{\hat\ell}}\|\theta\| +\sqrt{\frac{d}{n}}
+ \frac{1}{\|\theta\|}\left(\sqrt{\frac{1}{\hat{k}}}\sqrt{\frac{d}{n}}+\frac{d}{n}\right).
\label{lem: hat and theta diff}
\]
Plugging $\hat{k}\sim \frac{1}{\delta}  \wedge \frac{n}{d}$ and using $\|\theta\|^2 \geq \delta\vee Cd/n $ we get
$$
\left| \hat{s}-\|\theta\| \right| \lesssim \sqrt{\frac{1}{C}}\|\theta\| +\sqrt{\frac{d}{n}}.
$$
It comes out, since $C$ is large enough, that
$$
\left| \hat{s}-\|\theta\| \right| \leq \frac{\|\theta\|}{4}+\frac{\|\theta\|}{4}= \frac{\|\theta\|}{2}.$$
Hence with probability greater than $1-c_1e^{-c_0d}$, we  get
    $$\frac{\|\theta\|}{2}\leq\hat{s}\leq\frac{3\|\theta\|}{2}.$$
Now suppose $\delta \vee Cd/n>\|\theta\|^2 $, and using the control of $\|\hat\Sigma - \Sigma\|_{op}$, we get 
$$
\left| \hat{s}^2-\|\theta\|^2 \right| \lesssim \sqrt{\frac{d}{(n\delta)\vee (Cd)}}\|\theta\|^2 +\sqrt{\frac{d}{n}}\|\theta\|
+ \sqrt{\frac{(\delta\vee Cd/n) d}{n}}.
$$
Then we conclude that
\[
\left| \hat{s}^2-\|\theta\|^2 \right|\leq 2(\delta\vee Cd/n).
\label{eqn: second line square difference}
\]
It follows that $\hat{s}^2 < 3(\delta\vee Cd/n)$ with probability greater than $1-c_1e^{-c_0d}$. 

\subsection{Proof of Theorem \ref{thm: adaptive errors}}
\label{thm: adaptive errors proof}
For this proof, observe that $\hat \ell$ depends on the observations so we cannot apply directly Lemma \ref{lem:concentration of three terms} replacing $\ell$ by $\hat \ell$. Since we only need to know $\hat \ell$ up to a multiplicative factor, we will choose $\hat{\ell}$ on a logarithmic grid. Basically we will choose $\hat{\ell}= d2^{\hat m} \wedge n$ where $\hat{m}$ is the smallest integer such that $\hat{\ell} \geq \lceil n( 3\delta \vee \hat{s}^2) \vee Cd \rceil \wedge n$. Hence and in order to control the deviation terms in Lemma \ref{lem:concentration of three terms} we shall take a supremum for all values of $m$ that are no more than $\log(n/d)$. It comes out that Lemma \ref{lem:concentration of three terms} holds for $\hat\ell$ with probability $1-c_1\log(n/d)e^{-c_0d}$.

    We recall here that using Corollary \ref{cor: 3delta} we have, with probability $1-c_1e^{-c_0d}$, if $\hat{s}^2 <  3(\delta\vee Cd/n)$ we get that $\|\theta\|^2<12(\delta\vee Cd/n)$, and if $\hat{s}^2 \geq  3(\delta\vee Cd/n)$ we get that $ \|\theta\|^2 \geq \delta \vee Cd/n$.
    We analyse the error case by case.
    \begin{itemize}
        \item \textbf{Case 1:}
       $\|\theta\|^2\leq\sqrt\frac{\delta d}{n}\vee \frac{d}{n}$. 
       
       In this case, since $\sqrt\frac{\delta d}{n}\leq (\delta\vee Cd/n)$, then we know using Lemma \ref{lem: s and theta} that $\hat{s}^2 < 3\delta$ and $\ell \sim n\delta \vee d$. As shown before in \eqref{eq: saved from small values} it comes out that $\ell^2(\hat{\theta}(\hat \ell),\theta) \lesssim\sqrt\frac{\delta d}{n} + \frac{d}{n}$.
        \item \textbf{Case 2:} $\sqrt{\frac{\delta d}{n}}\vee \frac{d}{n} \leq \|\theta\|^2 < \delta \vee \frac{d}{n}$. \\ In this case again $\hat{s}^2 < 3\delta$ and $\ell \sim n\delta\vee d$, according to Corollary \ref{cor: 3delta}.
        Continuing with $\ell\sim n\delta \vee d $, and mimicking the corresponding case of Proposition \ref{prop: theoretical rates}, we get
        $$\ell(\hat{\theta}(\hat \ell),\theta) \lesssim \sqrt{\frac{d}{n}}\frac{\|\theta\|}{\sqrt{\delta+ Cd/n}} +\sqrt{\frac{d}{n}}
+ \frac{1}{\|\theta\|}\sqrt{\frac{(\delta \vee Cd/n) d}{n}} \lesssim \sqrt{\frac{d}{n}}+\frac{\sqrt{\frac{(\delta+Cd/n) d}{n}}}{\|\theta\|}\lesssim \frac{{\sqrt\frac{(\delta+d/n) d}{n}}}{\|\theta\|}.$$
        \item \textbf{Case 3:} $\delta\vee \frac{d}{n} \leq \|\theta\|^2 < 4$. 
        \\ In this case we have two options. Either $ 12 (\delta\vee Cd/n)\leq  \|\theta\|^2\leq4$ and hence $\hat{s}^2\geq 3(\delta\vee Cd/n)$ and $\ell\sim n\hat{s}^2$.
        That being the case, similarly to Proposition \ref{prop: theoretical rates}, by taking take $\ell \sim  n\|\theta\|^2$ replacing $\|\theta\|$ with $\hat{s}$ we derive 
        $$\ell(\hat{\theta}(\hat{\ell}),\theta) \lesssim \sqrt{\frac{d}{n\hat{s}^2}}\|\theta\| +\sqrt{\frac{d}{n}}
+ \frac{\hat{s}}{\|\theta\|}\sqrt{\frac{d}{n}}.$$
Using Lemma \ref{lem: s and theta}, we conclude that
$$\ell(\hat{\theta}(\hat{\ell}),\theta) \lesssim \sqrt{\frac{d}{n}} +\sqrt{\frac{d}{n}}
+ \sqrt{\frac{d}{n}}\lesssim\sqrt{\frac{d}{n}}.$$
In the remaining case where $\delta\vee \frac{d}{n}  \leq \|\theta\|^2\leq 12 (\delta\vee C\frac{d}{n}) $ we know that $\hat\ell$  will be of order $n\delta\vee d $ and the result is straightforward.

        \item Finally if $\|\theta\|^2 \geq 4$, then $\hat{s}^2\geq 1$ and $\hat\ell = n$. The result follows immediately.
        
    \end{itemize}
\subsection*{Proof of Theorem \ref{thm: unkown delta first loss bound}}
Repeating the same decomposition as before in \eqref{eqn: loss bound via operator norm} we get
\[
\ell(\tilde{\theta},\theta) \leq \left|\|\tilde{\theta}\|-\|\theta\|\right| + \|\theta\|\ell(\hat{v}_{max},v_{max}) \lesssim \frac{\|\hat{\Sigma}-\Sigma\|_{op}}{\|\theta\|}.
\]
The covariance decomposition is slightly different this time as we get
$$
\hat{\Sigma}-\Sigma=\frac{\theta\theta^T(\|\bar{\eta}\|^2-\mathbf{E}{\|\bar{\eta}\|^2} + \mathbf{E}{\|\bar{\eta}\|^2} - \ell)}{\ell} +\frac{w\bar{\eta}\cdot \theta^T+\theta \cdot(w\bar{\eta})^T}{\sqrt{k}\ell}+\frac{1}{k}\left(\frac{ww^T}{\ell}-\mathbf{I}_d\right).$$
Since $|\mathbf{E}{\|\bar{\eta}\|^2} - \ell| \leq n\delta$ as claimed in Lemma \ref{lem:expecation}, we get that 
$$\ell(\tilde{\theta},\theta)\lesssim \frac{\|\hat{\Sigma}-\Sigma\|_{op}}{\|\theta\|}+\frac{n\delta}{\ell}\|\theta\|.$$
Recalling our concentration bound as in \eqref{eqn: final loss bound} we finally derive that

\[\ell(\tilde{\theta},\theta) \lesssim \left(\sqrt{\frac{d}{\ell}}+\frac{n\delta}{\ell}\right)\|\theta\| +\sqrt{\frac{d}{n}}
+ \frac{1}{\|\theta\|}\left(\sqrt{\frac{1}{k}}\sqrt{\frac{d}{n}} +\frac{d}{n}\right),
\label{eqn: theta2 loss}
\]
with probability $1-c_1e^{-c_0d}$.

\subsection*{Proof of Proposition \ref{prop: rates for theta_2}}
Recall that $\ell^{**}=d\vee\left\lceil \left(n\delta\vee\frac{\delta^{2/3}n^{4/3}\|\theta\|^{4/3}}{d^{1/3}}\vee n \|\theta\|^2\right) \right\rceil\wedge n$. For the rest of the proof, and for simplicity, we denote $\tilde \theta(\ell^{**})$ by  $\tilde \theta$ and $\ell$ by $\ell^{**}$ .To begin with we analyse the loss rate when signal strength $\|\theta\|^2$ is small.
As we have already observed from the proof of Proposition \ref{prop: theoretical rates} we have that
\begin{equation}\label{eq:34}
    \ell^2(\tilde{\theta},\theta)  \lesssim
\sqrt{\frac{d}{n}} \|\theta\| + \left(\frac{d}{n}+\frac{1}{k}\sqrt{\frac{d}{\ell}}\right) + \left(1+\sqrt{\frac{d}{\ell}}\right)\|\theta\|^2,
\end{equation}
for $\|\theta\|^2 \leq \sqrt{\frac{\delta d}{n}}\vee \frac{d}{n}$. Observe that in that case $\ell^{**}\sim n\delta \vee d$ and hence
\begin{align*}
\ell^2(\tilde{\theta},\theta)  &\lesssim 
\sqrt{\frac{d}{n}} \|\theta\| + \left(\frac{d}{n}+\delta\sqrt{\frac{d}{n\delta}}\right) + \left(1+\sqrt{\frac{d}{n\delta}}\right)\|\theta\|^2 \\ &= 
\sqrt{\frac{d}{n}} \|\theta\| + \frac{d}{n} + \sqrt{\frac{\delta d}{n}} + \|\theta\|^2 + \sqrt{\frac{d}{n\delta}}\|\theta\|^2 \\ &\leq   \sqrt{\frac{\delta d}{n}} + \frac{d}{n}.
\end{align*}
\\
Next, let's consider two cases for $\delta$. For the sake of convenience in the subsequent lines the bounds are derived for the loss function itself, not for the square of it.
\begin{itemize}
    \item Case 1: $\delta\leq\sqrt{\frac{d}{n}}$ \\

Here, in the scenario when $\sqrt{\frac{\delta d}{n}}\vee \frac{d}{n} \leq \|\theta\|^2 \leq \frac{\delta^2n}{d}\vee \frac{d}{n}$ we continue the process with the following number of buckets: 
$$\ell\sim \frac{\delta^{\frac{2}{3}}n^{\frac{4}{3}}\|\theta\|^{\frac{4}{3}}}{d^{\frac{1}{3}}}.
$$
Observe that $\ell \leq n$ as $\delta \leq \sqrt{d/n}$.
Now let's analyse the loss bound \eqref{eqn: theta2 loss} with this choice of $\ell$.
$$\sqrt{\frac{d}{\ell}} \sim \frac{d^{2/3}}{\delta^{1/3}n^{2/3}\|\theta\|^{2/3}}
,\;\; \frac{n\delta}{\ell} \sim \frac{\delta^{1/3}d^{1/3}}{n^{1/3}\|\theta\|^{4/3}},\;\; \sqrt{\frac{d\ell}{n^2}} \sim \frac{\delta^{1/3}d^{1/3}\|\theta\|^{2/3}}{n^{1/3}}.$$
Hence
$$\ell(\tilde{\theta},\theta) \lesssim \frac{d^{2/3}\|\theta\|^{1/3}}{\delta^{1/3}n^{2/3}}+\frac{\delta^{1/3}d^{1/3}}{n^{1/3}\|\theta\|^{1/3}} +\sqrt{\frac{d}{n}}\lesssim\left(\frac{\frac{\delta d}{n}}{\|\theta\|}\right)^{\frac{1}{3}}.
$$
Next, if $\frac{\delta^2n}{d}\vee \frac{d}{n} \leq \|\theta\|^2\leq 1$ we continue with $$\ell\sim n\|\theta\|^2,$$ achieving the parametric rate
$$\ell(\tilde{\theta},\theta)\lesssim\sqrt{\frac{d}{n}}+\frac{\delta}{\|\theta\|} +\sqrt{\frac{d}{n}}
+ \sqrt{\frac{d}{n}}\lesssim\sqrt{\frac{d}{n}}.$$
Lastly when $\|\theta\|^2\geq 1$ the parametric rate can be recovered with $\ell=n$.\\

\item Case 2  : $\sqrt{\frac{d}{n}}\leq\delta$\\
The first scenario is when $\sqrt{\frac{\delta d}{n}} \leq \|\theta\|^2 \leq \frac{\sqrt{\frac{d}{n}}}{\delta}$. Here we have that
$$\ell \sim \frac{\delta^{\frac{2}{3}}n^{\frac{4}{3}}\|\theta\|^{\frac{4}{3}}}{d^{\frac{1}{3}}}.
$$
Again, plugging in $\ell$ in \eqref{eqn: theta2 loss} and using $\sqrt{\frac{d}{n}}\leq\delta$ we get moreover that
$$\ell(\tilde{\theta},\theta) \lesssim \frac{d^{2/3}\|\theta\|^{1/3}}{\delta^{1/3}n^{2/3}}+\frac{\delta^{1/3}d^{1/3}}{n^{1/3}\|\theta\|^{1/3}} +\sqrt{\frac{d}{n}}\leq \left(\frac{\delta^{1/3}d^{1/3}}{n^{1/3}\|\theta\|^{1/3}} +\sqrt{\frac{d}{n}}
    \right)\lesssim\left(\frac{\frac{\delta d}{n}}{\|\theta\|}\right)^{\frac{1}{3}}.
$$
Furthermore, if $\frac{\sqrt{\frac{d}{n}}}{\delta} \leq \|\theta\|^2$ with the choice of $\ell_3=n$ we recover  the usual rate of the vanilla spectral method,
$$\ell(\tilde{\theta},\theta) \lesssim \sqrt{\frac{d}{n}}
+ \frac{1}{\|\theta\|}\sqrt{\frac{d}{n}}\lesssim\begin{cases}
    \frac{\sqrt{\frac{d}{n}}}{\|\theta\|},  & \|\theta\|^2 \leq 1\\
\sqrt{\frac{d}{n}}, & 1 \leq \|\theta\|^2. \\
\end{cases}$$
In both cases the compact formula for $\ell$ is given by
$$\ell^{**}\sim d\vee \left(n\delta\vee\frac{\delta^{2/3}n^{4/3}\|\theta\|^{4/3}}{d^{1/3}}\right)\wedge n.$$
\end{itemize}
\begin{rem}
It is rather convenient to have one compact form for the number of buckets in both cases of Proposition \ref{prop: rates for theta_2}. In that respect, let's notice that in Case 2 when $\|\theta\|\leq1$ then $\ell^{**} \geq n\|\theta\|^2 $. Indeed, since $\delta \geq d/n$, we have that either
$$\ell_1=n\delta \geq n\sqrt{\frac{\delta d}{n}}\geq n\|\theta\|^2,$$
or
$$\ell_2=\frac{\delta^{\frac{2}{3}}n^{\frac{4}{3}}\|\theta\|^{\frac{4}{3}}}{d^{\frac{1}{3}}}\geq\frac{d^{\frac{1}{3}}n^{\frac{4}{3}}\|\theta\|^{\frac{4}{3}}}{n^{\frac{1}{3}}d^{\frac{1}{3}}}=n\|\theta\|^{\frac{4}{3}}\geq n\|\theta\|^2.$$
This means that in Case 2 taking the maximum with additional $n\|\theta\|^2$ won't change anything
$$\left(n\delta\vee\frac{\delta^{2/3}n^{4/3}\|\theta\|^{4/3}}{d^{1/3}}\right)=\left(n\delta\vee\frac{\delta^{2/3}n^{4/3}\|\theta\|^{4/3}}{d^{1/3}}\vee n\|\theta\|^2\right).$$
Therefore we can say that $\ell^{**}$ takes the following general form 
$$\ell^{**}\sim d \vee \left(n\delta\vee\frac{\delta^{2/3}n^{4/3}\|\theta\|^{4/3}}{d^{1/3}}\vee n \|\theta\|^2\right)\wedge n.$$
\end{rem}
As a consequence of this, let's declare that with this choice we have one that either $\ell^{**}=n$ or the following three inequalities hold together
\[\ell^{**}\geq n\delta \vee d, \;\;\;\;\;\; \ell^{**}\geq\frac{\delta^{2/3}n^{4/3}\|\theta\|^{4/3}}{d^{1/3}}, \;\;\;\;\;\; \ell^{**} \geq n \|\theta\|^2 \label{in: l has three bounds}. \]
This observation will be useful later on.

\subsection{Proof of Theorem \ref{thm: lepski}}
 For this proof, let us denote $\tilde\theta:=\tilde\theta(\tilde \ell)$, $\tilde{\theta}_k := \tilde\theta(g_k)$ and $\omega^2_k:=4 C^2 \left(\frac{\sqrt{d g_k}}{n}\wedge\frac{d g_k}{n^2\|\tilde\theta(g_k) \|^2 }\right)$.
 We have for any $\omega$,
    \[
\mathbf{P}(\ell(\tilde{\theta},\theta)\geq \omega)=
\mathbf{P}\left(\ell(\tilde{\theta},\theta)\geq \omega \; \cap \; \{\tilde{\ell}\leq\ell^{**}\}\right) + \mathbf{P}\left(\ell(\tilde{\theta},\theta)\geq \omega \; \cap \;\{ \tilde{\ell}>\ell^{**}\}\right)\dfn I_1 + I_2.
\]
We also recall using \eqref{eq:34} that for $g_k \geq n\delta$ we have
\[
    \left| \|\tilde\theta(g_k) \|^2 - \|\theta\|^2\right|  \lesssim
\sqrt{\frac{d}{n}} \|\theta\| + \left(\frac{d}{n}+\sqrt{\frac{g_k d}{n^2}}\right) + \left(\frac{n\delta}{g_k}+\sqrt{\frac{d}{g_k}}\right)\|\theta\|^2,
\]
with probability $1-c_1\cdot e^{-c_0d}$.

\textit{Bounding $I_1$:}
We begin by defining the smallest element $g_{m_{*}}$ from grid $G$ that is greater than $\ell^{**}$. Our intention here is to replace the unknown $\ell^{**}$ with a $\ell$ in a close range, for which further analysis is tractable.\\ 
Formally, 
$$m_*\dfn \min\{m\in\{0,1,\dots,M\}\;:\; \ell^{**}\leq g_m\}.$$ 
By this definition $g_{m_*-1}<\ell^{**}$, hence $g_{m_*}$ differs from $\ell^{**}$ by a factor of 2 at most, since $g_{m_*}=2g_{m_*-1}<2\ell^{**}$. Let $\epsilon_*$ and $ \epsilon_k$ be signs such that $\ell(\tilde{\theta}_{m_*},\theta) =\|\epsilon_{m^*}\tilde{\theta}_{m_*}-\theta\| $ and for all $\tilde{m}+1 \leq k\leq m_*$ we have $\ell(\tilde{\theta}_{k},\tilde{\theta}_{k-1}) =\|\epsilon_{k}\tilde{\theta}_{k}-\epsilon_{k-1}\tilde{\theta}_{k-1}\|$. 

We decompose the difference using triangle inequality by stepping through grid values until we reach the first number of buckets $g_{m_*}$.
$$\ell(\tilde\theta,\theta) \leq \|\tilde\theta -\epsilon_{\tilde m} \theta\| \leq \|\epsilon_{\tilde m} \tilde{\theta}_{g_{\tilde{m}}}-  \theta\| \leq\sum_{k=\tilde{m}+1}^{m_*}\|\epsilon_{ k}\tilde{\theta}_{k}-\epsilon_{k-1}\tilde{\theta}_{k-1}\|+ \|\epsilon_{m^*}\tilde{\theta}_{m_*}-\theta\| \leq\sum_{k=1}^{m_*}\omega_k+\ell(\tilde{\theta}_{m_*},\theta).$$
In order to deal with the sum $\sum_{k=1}^{m_*}\omega_k$ we decompose the sum into two sums $\sum_{k=1}^{m_1}\omega_k$ and $\sum_{k=m_1+1}^{m_*}\omega_k$, where $g_{m_1} \sim n\delta$. It is straightforward, using gemetric series, that 
$$
\left(\sum_{k=1}^{m_1}\omega_k\right)^2 \lesssim  \sqrt{\frac{d\delta}{n}}+\frac{d}{n}.
$$
It remains to control $\sum_{k=m_1+1}^{m_*}\omega_k$ where all $g_k \geq n\delta$.

On the one hand, when $\|\theta\|^2 \leq \frac{\sqrt{d\ell^{**}}}{n} $, we have that $w^2_k \leq 4C^2\frac{\sqrt{d g_k}}{n}$ and the sum $\sum_{k=\tilde{m}+1}^{m_*}\omega_k$ can be upper bounded by a sum of a geometric progression, and so
$$\sum_{k=\tilde{m}+1}^{m_*}\omega_k\leq\sum_{k=1}^{m_*}\omega_k=2C\frac{\sqrt[4]{d g_1}}{\sqrt{n}}\sum_{k=1}^{m_*}(\sqrt[4]{2})^k=2C\frac{\sqrt[4]{d g_1}}{\sqrt{n}}\frac{(\sqrt[4]{2})^{m_*+1}-1}{\sqrt[4]{2}-1}\leq 2C\frac{\sqrt[4]{d g_{m_*+1}}}{\sqrt{n}} .$$
It comes out that 
$$
\left(\sum_{k=m_1}^{m_*}\omega_k\right)^2 \lesssim \left( \sqrt{\frac{\delta d}{n}} + \frac{d}{n}\right).
$$
On the other hand, when $ \frac{\sqrt{d\ell^{**}}}{n}  \leq \|\theta\|^2 \leq 1$, we have, with probability greater than $1-c_1\cdot e^{-c_0d}$, that
$$
\|\tilde\theta(g_k) \|^2 \gtrsim \|\theta\|^2.
$$
It comes out that $w^2_k \lesssim 4C^2\frac{d g_k}{n^2\|\theta\|^2}$ and the sum $\sum_{k=\tilde{m}+1}^{m_*}\omega_k$ can be upper bounded by a sum of a geometric progression, and so
$$
\left(\sum_{k=m_1}^{m_*}\omega_k\right)^2 \lesssim \frac{d\ell^{**}}{n^2\|\theta\|^2}.
$$
It comes out that 
$$
\left(\sum_{k=m_1}^{m_*}\omega_k\right)^2 \lesssim \frac{d\ell^{**}}{n^2\|\theta\|^2} \wedge \left( \sqrt{\frac{\delta d}{n}} + \frac{d}{n}\right).
$$
We have by Proposition \ref{prop: rates for theta_2}, that, with probability greater than $1-c_1e^{-c_0 d}$,
$$
\ell^2(\tilde{\theta}_{m_*},\theta) \lesssim \left(\sqrt{\frac{\delta d }{n}} + \frac{d}{n}\right) \wedge\left(\frac{\ell^{**} d}{n\|\theta\|^2} + \frac{d}{n}\right).$$
Hence we get that as long as $\tilde{\ell} \leq \ell^{**}$, then with probability greater than $1-c_1\cdot e^{-c_0d}$, we have
\begin{equation*}
\ell^2(\tilde\theta,\theta) \lesssim \left(\sqrt{\frac{\delta d }{n}} + \frac{d}{n}\right).
\end{equation*}

\textit{Bounding $I_2$:}
In this case $g_k \geq \ell^{**}$. Notice first that, when $\|\theta\|^2 \leq \frac{\sqrt{d g_k}}{n} $, we have that $\|\tilde\theta(g_k) \|^2 \lesssim \frac{\sqrt{d g_k}}{n}$ and hence $w^2_k \gtrsim \frac{\sqrt{d g_k}}{n}$ and when $\|\theta\|^2 \geq \frac{\sqrt{dg_k}}{n} $ then $
\|\tilde\theta(g_k) \|^2 \lesssim \|\theta\|^2
$
and $w^2_k \gtrsim \frac{d g_k}{n^2\|\theta\|^2}$. It comes out that as long as $g_k \geq \ell^{**}$ we have that
$$
w_k^2 \gtrsim  \frac{\sqrt{d g_k}}{n} \wedge \frac{d g_k}{n^2\|\theta\|^2}.
$$

In this case
\begin{align*}
    I_2 \leq \mathbf{P}(\tilde{\ell}>\ell^{**})&=\mathbf{P}(\tilde{m}>\ell^{**})\\&
    \leq 
    \sum_{m:\;g_m>\ell^{**}}\mathbf{P}(\tilde{\ell}=g_m) \\&
    \leq
    \sum_{m:\;g_m>\ell^{**}} \mathbf{P}(\ell(\tilde{\theta}_{m},\tilde{\theta}_{m-1})>\omega_m)\\ &
    \leq \sum_{m:\;g_m>\ell^{**}}\mathbf{P}(\ell(\tilde{\theta}_{m},\theta) + \ell(\tilde{\theta}_{m-1},\theta)>\omega_m)\\ &
    \leq 2\sum_{m:\;g_m\geq\ell^{**}}\mathbf{P}\left(\ell(\tilde{\theta}_{m},\theta) >\frac{\omega_m}{2}\right)\\ &
    \stackrel{(*)}{\lesssim}
    2|G|e^{-c_0d}\\&
    \lesssim\log\left(\frac{n}{d}\right)e^{-c_0d}.
\end{align*}
where $(*)$ holds, because for every $g_m\geq\ell^{**}$ due to Lemma \ref{lem: property for large ell} we have $$\ell^2(\tilde{\theta}_{g_{m}},\theta)\lesssim w^2_k \;\;\;\;\;\;\text{w.p.}\;\;\; 1-c_1\cdot e^{-c_0 d }.$$

\subsection{Proof of Theorem \ref{thm: lepski2}}
 For this proof we will mimick the proof of Theorem \ref{thm: lepski}. Let us denote $\tilde\theta:=\tilde\theta(\tilde{\ell}_2)$, $\tilde{\theta}_k := \tilde\theta(g_k)$ and $\omega^2_k:=4 C^2 \left(\frac{\sqrt{d g_k}}{n}\wedge\frac{d g_k}{n^2\|\tilde\theta(\tilde{\ell}) \|^2 }\right)$.
Notice that, with probability $1-c_1\log(n/d) e^{-c_0d}$, we have  
$$
\left|\|\tilde\theta(\tilde{\ell}) \| - \|\theta\|\right|^2 \lesssim\frac{d}{n}+ \sqrt{\frac{d \delta}{n}}.
$$

 We have for any $\omega$,
    \[
\mathbf{P}(\ell(\tilde{\theta},\theta)\geq \omega)=
\mathbf{P}\left(\ell(\tilde{\theta},\theta)\geq \omega \; \cap \; \{\tilde{\ell}_2\leq\ell^{**}\}\right) + \mathbf{P}\left(\ell(\tilde{\theta},\theta)\geq \omega \; \cap \;\{ \tilde{\ell}_2>\ell^{**}\}\right)\dfn I_1 + I_2.
\]

\textit{Bounding $I_1$:}
We begin by defining the smallest element $g_{m_{*}}$ from grid $G$ that is greater than $\ell^{**}$. Our intention here is to replace the unknown $\ell^{**}$ with a $\ell$ in a close range, for which further analysis is tractable.\\ 
Formally, 
$$m_*\dfn \min\{m\in\{0,1,\dots,M\}\;:\; \ell^{**}\leq g_m\}.$$ 
By this definition $g_{m_*-1}<\ell^{**}$, hence $g_{m_*}$ differs from $\ell^{**}$ by a factor of 2 at most, since $g_{m_*}=2g_{m_*-1}<2\ell^{**}$. Let $\epsilon_*$ and $ \epsilon_k$ be signs such that $\ell(\tilde{\theta}_{m_*},\theta) =\|\epsilon_{m^*}\tilde{\theta}_{m_*}-\theta\| $ and for all $\tilde{m}+1 \leq k\leq m_*$ we have $\ell(\tilde{\theta}_{k},\tilde{\theta}_{k-1}) =\|\epsilon_{k}\tilde{\theta}_{k}-\epsilon_{k-1}\tilde{\theta}_{k-1}\|$. 

We decompose the difference using triangle inequality by stepping through grid values until we reach the first number of buckets $g_{m_*}$.
$$\ell(\tilde\theta,\theta) \leq \|\tilde\theta -\epsilon_{\tilde m} \theta\| \leq \|\epsilon_{\tilde m} \tilde{\theta}_{g_{\tilde{m}}}-  \theta\| \leq\sum_{k=\tilde{m}+1}^{m_*}\|\epsilon_{ k}\tilde{\theta}_{k}-\epsilon_{k-1}\tilde{\theta}_{k-1}\|+ \|\epsilon_{m^*}\tilde{\theta}_{m_*}-\theta\| \leq\sum_{k=1}^{m_*}\omega_k+\ell(\tilde{\theta}_{m_*},\theta).$$

On the one hand, when $\|\theta\|^2 \leq \frac{\sqrt{d\ell^{**}}}{n} $, we have that $w^2_k \leq 4C^2\frac{\sqrt{d g_k}}{n}$ and the sum $\sum_{k=\tilde{m}+1}^{m_*}\omega_k$ can be upper bounded by a sum of a geometric progression, and so
$$\sum_{k=\tilde{m}+1}^{m_*}\omega_k\leq\sum_{k=1}^{m_*}\omega_k=2C\frac{\sqrt[4]{d g_1}}{\sqrt{n}}\sum_{k=1}^{m_*}(\sqrt[4]{2})^k=2C\frac{\sqrt[4]{d g_1}}{\sqrt{n}}\frac{(\sqrt[4]{2})^{m_*+1}-1}{\sqrt[4]{2}-1}\leq 2C\frac{\sqrt[4]{d g_{m_*+1}}}{\sqrt{n}} .$$
It comes out that 
$$
\left(\sum_{k=m_1}^{m_*}\omega_k\right)^2 \lesssim \left( \sqrt{\frac{\delta d}{n}} + \frac{d}{n}\right).
$$
On the other hand, when $ \frac{\sqrt{d\ell^{**}}}{n}  \leq \|\theta\|^2 \leq 1$, we have, with probability greater than $1-c_1\log(n/d) e^{-c_0d}$, that
$$
\|\tilde\theta(\tilde{\ell}) \|^2 \lesssim \|\theta\|^2.
$$
It comes out that $w^2_k \lesssim 4C^2\frac{d g_k}{n^2\|\theta\|^2}$ and the sum $\sum_{k=\tilde{m}+1}^{m_*}\omega_k$ can be upper bounded by a sum of a geometric progression, and so
$$
\left(\sum_{k=m_1}^{m_*}\omega_k\right)^2 \lesssim \frac{d\ell^{**}}{n^2\|\theta\|^2}.
$$
It comes out that 
$$
\left(\sum_{k=m_1}^{m_*}\omega_k\right)^2 \lesssim \frac{d\ell^{**}}{n^2\|\theta\|^2} \wedge \left( \sqrt{\frac{\delta d}{n}} + \frac{d}{n}\right).
$$
We have by Proposition \ref{prop: rates for theta_2}, that, with probability greater than $1-c_1e^{-c_0 d}$,
$$
\ell^2(\tilde{\theta}_{m_*},\theta) \lesssim \left(\sqrt{\frac{\delta d }{n}} + \frac{d}{n}\right) \wedge\left(\frac{\ell^{**} d}{n\|\theta\|^2} + \frac{d}{n}\right).$$
Hence we get that as long as $\tilde{\ell} \leq \ell^{**}$, then with probability greater than $1-c_1\log(n/d) e^{-c_0d}$, we have
\begin{equation*}
\ell^2(\tilde\theta,\theta) \lesssim \left(\sqrt{\frac{\delta d }{n}} + \frac{d}{n}\right) \wedge\left(\frac{\ell^{**} d}{n\|\theta\|^2} + \frac{d}{n}\right).
\end{equation*}

\textit{Bounding $I_2$:}
In this case $g_k \geq \ell^{**}$. Notice first that, when $\|\theta\|^2 \leq \frac{\sqrt{d g_k}}{n} $, we have that $\|\tilde\theta(\tilde{\ell}) \|^2 \lesssim \frac{\sqrt{d g_k}}{n}$ and hence $w^2_k \gtrsim \frac{\sqrt{d g_k}}{n}$ and when $\|\theta\|^2 \geq \frac{\sqrt{dg_k}}{n} $ then $
\|\tilde\theta(\tilde{\ell}) \|^2 \lesssim \|\theta\|^2
$
and $w^2_k \gtrsim \frac{d g_k}{n^2\|\theta\|^2}$. It comes out that as long as $g_k \geq \ell^{**}$ we have that
$$
w_k^2 \gtrsim  \frac{\sqrt{d g_k}}{n} \wedge \frac{d g_k}{n^2\|\theta\|^2}.
$$

In this case
\begin{align*}
    I_2 \leq \mathbf{P}(\tilde{\ell}_2>\ell^{**})&=\mathbf{P}(\tilde{m}>\ell^{**})\\&
    \leq 
    \sum_{m:\;g_m>\ell^{**}}\mathbf{P}(\tilde{\ell}_2=g_m) \\&
    \leq
    \sum_{m:\;g_m>\ell^{**}} \mathbf{P}(\ell(\tilde{\theta}_{m},\tilde{\theta}_{m-1})>\omega_m)\\ &
    \leq \sum_{m:\;g_m>\ell^{**}}\mathbf{P}(\ell(\tilde{\theta}_{m},\theta) + \ell(\tilde{\theta}_{m-1},\theta)>\omega_m)\\ &
    \leq 2\sum_{m:\;g_m\geq\ell^{**}}\mathbf{P}\left(\ell(\tilde{\theta}_{m},\theta) >\frac{\omega_m}{2}\right)\\ &
    \stackrel{(*)}{\lesssim}
    2|G|e^{-c_0d}\\&
    \lesssim\log\left(\frac{n}{d}\right)e^{-c_0d}.
\end{align*}
where $(*)$ holds, because for every $g_m\geq\ell^{**}$ due to Lemma \ref{lem: property for large ell} we have $$\ell^2(\tilde{\theta}_{g_{m}},\theta)\leq w_k^2 \;\;\;\;\;\;\text{w.p.}\;\;\; 1-c_1\cdot e^{-c_0 d }.$$

\section{Technical Lemmas}
\begin{lem}
    Assume that $k\geq 1$. Let $\delta,\delta' \in (0,1/2)$, then we have that 
    $$
    |g(\delta) - g(\delta')| \leq \frac{2k|\delta-\delta'|}{3}.
    $$
\end{lem}
\begin{proof}
    Let us denote $\rho:= 1-2\delta$. It is easy to observe that $\mathbf{E}(X_1) = \rho$. For $0<\delta $, we have
    \begin{align*}
        k^2 g(\delta) &= \mathbf{E}\left( \sum_{s=1}^k \prod_{i=1}^s X_i\right)^2 \\
        &= 2  \sum_{1 \leq s<s'\leq k} \mathbf{E}\left(\prod_{i=1}^s X_i^2\prod_{i=s+1}^{s'} X_i\right) + \sum_{1 \leq s \leq k} \mathbf{E}\left(\prod_{i=1}^s X_i^2\right)\\
        &= k + 2  \sum_{1 \leq s<s'\leq k} \rho^{s'-s}.
    \end{align*}
It comes out that $g(\delta)$ is non-increasing as a function of $\delta$. Hence 
$$
|g(\delta) - g(\delta')| \leq \left|\underset{\delta \to 0}{\text{lim}}g'(\delta)\right| |\delta-\delta'|.
$$
Moreover we have that 
    $$
    k^2g(\delta)= k + 2\sum_{s'=2}^{k} \frac{\rho - \rho^{s'}}{1-\rho} = \frac{k(1+\rho)}{1-\rho} - 2\frac{\rho - \rho^{k+1}}{(1-\rho)^2}=\frac{2k(1-\delta)\delta - 1+2\delta + (1-2\delta)^{k+1}}{ 2\delta^2}.
    $$
    Using Taylor series approximation, we have that 
    $$
    (1-2\delta)^{k+1} = 1 - 2(k+1)\delta +2k(k+1)\delta^2 - 4k(k+1)(k-1)\delta^3/3 +o(\delta^3).
    $$
    Hence 
    $$
    g(\delta) = 1 - \frac{2(k^2-1)}{3k}\delta+o(\delta).
    $$
    It comes out that 
    $$
    \left|\underset{\delta \to 0}{\text{lim}}g'(\delta)\right| \leq 2k/3.
    $$
    This concludes the proof.
\end{proof}

\begin{lem}
\label{lem:concentration of three terms}
Assume that $d \leq \ell \leq n$, $n\delta\leq\ell$ and $\xi$ a sub-Gaussian isotropic vector. Then
$$\P\left( | \|\bar{\eta}\|^2 - \mathbf{E}(\|\bar{\eta}\|^2) | \geq \sqrt{d\ell} \right)\geq 1-2e^{-cd},$$
$$\P\left( \frac{\ell}{4}\leq \|\bar{\eta}\|^2 \leq \ell \right)\geq 1-e^{-c\ell},$$
$$\P\left(\left\|\frac{ww^T}{\ell}-I_d\right\|_{op} \leq C\frac{d}{\ell}+C\sqrt{\frac{d}{\ell}} \right) \geq 1 - e^{-cd},$$
and
$$\P\left(\|\xi\|\leq C\sqrt{d}\right) \geq 1-e^{-cd},$$
for absolute constants $c>0$ and $C>0$.
\end{lem}

\begin{proof}
First, since $0\leq\bar{\eta}_i^2\leq1$ and $\bar{\eta}_i^2$ are i.i.d., Hoeffding's inequality can be applied for $-\bar{\eta}_i^2$-s, namely for every $t>0$

\begin{equation}
    \P\left(|\|\bar{\eta}\|^2-\mathbb{E}\|\bar{\eta}\|^2|\geq t \right)\leq 2e^{\frac{-t^2}{2\ell}}.
    \label{eqn: hof1}
\end{equation}
Hence we get the first inequality for $t=\sqrt{\frac{d}{\ell}}$. Moreover
\begin{equation}
    \P\left(\|\bar{\eta}\|^2 > \mathbb{E}\|\bar{\eta}\|^2 -t \right)\geq 1-e^{\frac{-t^2}{2\ell}}
\end{equation}
Recalling Lemma \ref{lem:expecation}, we have that $\mathbb{E}[\bar{\eta}_i^2]\geq1-2\delta k/3$ and for $\delta k\leq 1$ (or $\ell \geq \delta n$) we get $\mathbb{E}[\bar{\eta}_i^2]\geq\frac{1}{3}$ and $\mathbb{E}[\|\bar{\eta}\|^2]\geq\frac{\ell}{3}$. Consequently, taking $t=\frac{\ell}{12}$ we obtain that with probability greater than $1-e^{\frac{-\ell}{288}}$

$$\|\bar{\eta}\|^2\geq\frac{\ell}{4}$$
On the other hand (with probability 1) $$\|\bar{\eta}\|^2 = \sum_{i=1}^\ell \bar{\eta}_i^2\leq\ell$$
Which summarises the first concentration.

Second, by sub-Gaussian covariance estimation from \citep{vershynin2010introduction}, for every $t>0$ with probability greater than $1-\exp(-t)$ it holds that
$$\left\|\frac{ww^T}{\ell}-I_d\right\|_{op} \leq C\left(\frac{t}{\ell} + \frac{d}{\ell}+\sqrt{\frac{t}{\ell}} +\sqrt{\frac{d}{\ell}}\right) $$
By taking $t= d$ we obtain the desired result.

Third, using Hanson-Wright inequality from \citep{rudelson2013hanson} for sub-Gaussian vectors we get the following tail bound for every $t>0$
$$\P\left(\|\xi\|^2\geq d +C(t + \sqrt{dt})\right) \leq e^{-t}$$
Then taking $t=d$ we conclude our third concentration.
\end{proof}

\subsection*{Proof of Lemma \ref{lem: property for large ell}}

Let's consider two scenarios for  $\|\theta\|^2$.\\
For a small signal strength $\|\theta\|^2\leq\frac{\sqrt{d\ell}}{n}$, we undertake analogous bounding procedure as in Proposition \ref{prop: theoretical rates}. 
\begin{align*}
    \ell^2(\tilde{\theta},\theta) &\lesssim 
\|\hat{\theta}\|^2+\|\theta\|^2 \\&\lesssim 
\left|\|\hat{\theta}\|^2-\|\theta\|^2 \right| + \|\theta\|^2 \\& \lesssim \|\hat{\Sigma}-\Sigma\|_{op}+\|\theta\|^2 \\& \lesssim 
\sqrt{\frac{d}{n}} \|\theta\| + \left(\frac{d}{n}+\frac{\sqrt{d\ell}}{n}\right) + \left(1+\sqrt{\frac{d}{\ell}}\right)\|\theta\|^2
\\& \lesssim \sqrt{\frac{d}{n}}\frac{\sqrt[4]{d\ell}}{\sqrt{n}} + \frac{\sqrt{d\ell}}{n}+\left(1+\sqrt{\frac{d}{\ell}}\right)\frac{\sqrt{d\ell}}{n} \\& 
\lesssim \frac{\sqrt{d\ell}}{n} + \frac{\sqrt{d\ell}}{n} + \frac{\sqrt{d\ell}}{n}\\& \lesssim \frac{\sqrt{d\ell}}{n} \wedge \left(\frac{d\ell}{\|\theta\|^2 n^2} + \frac{d}{n}\right),
\end{align*}
where, in the last inequality, we use the fact that $\frac{d}{n}\leq \frac{\sqrt{d\ell}}{n} \leq\frac{d\ell}{\|\theta\|^2 n^2}$ in this regime .

For further analysis let's remind ourselves that $\ell\geq\ell^{**}$ and $\|\theta\|^2\geq\frac{\sqrt{d\ell}}{n}$. Hence $\frac{\sqrt{d\ell}}{n} \geq \frac{d\ell}{n\|\theta\|^2} + \frac{d}{n}$ because $d\leq \ell$.
According to Theorem \ref{thm: unkown delta first loss bound} we know that $$\ell^2(\tilde{\theta},\theta) \lesssim \left(\frac{d}{\ell}+\frac{n^2\delta^2}{\ell^2}\right)\|\theta\|^2 +\frac{d}{n}
+ \frac{1}{\|\theta\|^2}\frac{d\ell}{n^2}.
$$


It remains to show that 
$$\left(\frac{d}{\ell} + \frac{n^2\delta^2}{\ell^2}\right)\|\theta\|^2\leq\frac{d}{n}
+ \frac{1}{\|\theta\|^2}\frac{d\ell}{n^2},$$ in the case where $\|\theta\|^2\leq 1$, since when $\|\theta\|^2\geq 1$ we have that $\ell^{**}=n$ and 
$$
\ell^2(\tilde{\theta},\theta) \lesssim \frac{d}{n}.
$$
To prove the latter let's observe that , using \eqref{in: l has three bounds}, for $\ell \geq \ell^{**}$ we have
\[\ell\geq n\delta \vee d, \;\;\;\;\;\; \ell\geq\frac{\delta^{2/3}n^{4/3}\|\theta\|^{4/3}}{d^{1/3}}, \;\;\;\;\;\; \ell\geq n \|\theta\|^2 . \]
It comes out that 
$$
\frac{n^2\delta^2}{\ell^2}\|\theta\|^2\leq \frac{1}{\|\theta\|^2}\frac{d\ell}{n^2},
$$
and that
$$
\frac{d}{\ell}\|\theta\|^2\leq \frac{1}{\|\theta\|^2}\frac{d\ell}{n^2}.
$$
 Thus
$$\ell^2(\tilde{\theta},\theta)\lesssim\frac{d \ell }{n^2 \|\theta\|^2}.$$
In summary we always have that
$$\ell^2(\tilde{\theta},\theta)\lesssim \frac{\sqrt{d\ell}}{n}\wedge\left( \frac{d \ell }{n^2 \|\theta\|^2} + \frac{d}{n}\right).$$

\end{document}